\documentclass[11pt]{amsart}

\usepackage{enumerate,color,xcolor}

\newtheorem{lem}{Lemma}[section]

\newtheorem{thm}{Theorem}[section]
\newtheorem{cor}{Corollary}[section]

\newtheorem{rmk}{Remark}[section] 

\numberwithin{equation}{section}

\def\<{\langle}
\def\>{\rangle}

\title [CLT and MDP for the mean fields of Hawkes processes]
{Fluctuations and moderate deviations for the mean fields of Hawkes processes}

\author{Fuqing GAO}
\address{School of Mathematics and Statistics, Wuhan University, Wuhan 430072, China}
\email{fqgao@whu.edu.cn}
 \thanks{Supported by NSFC Grants 11971361 and 11731012.}

\author{Yunshi Gao}
\address{School of Mathematics and Statistics, Wuhan University, Wuhan
430072, China}  \email{yunshig@whu.edu.cn}

\author{Lingjiong Zhu}
\address{Department of Mathematics, Florida State University, 1017 Academic Way, Tallahassee, FL-32306, United States of America}
\email{zhu@math.fsu.edu}

\thanks{Supported by  NSF Grants DMS-2053454 and DMS-2208303.}

\date{July 29, 2023.}

\begin{document}



\begin{abstract} The Hawkes process is a counting process that has self- and mutually-exciting features
with many applications in various fields. In recent years, there have been many interests in 
the mean-field results of the Hawkes process and its extensions.
It is known that the mean-field limit of  a multivariate nonlinear Hawkes process is a time-inhomogeneous Poisson process.
In this paper, we study the fluctuations for the  mean fields and  the large deviations associated with the fluctuations, i.e., the moderate deviations.
\end{abstract}

\maketitle

\section{Introduction}
 
The Hawkes process is a self- and mutually-exciting counting process that has clustering effect, 
with wide applications in criminology, econometrics, finance, genome analysis, insurance, machine learning, marketing,
neuroscience, operations management, queueing theory,  sociology,  seismology and many other fields.
The (multivariate linear) Hawkes process  was first proposed by Hawkes \cite{Hawkes1971Biometrika}  to model earthquakes and their aftershocks.
The extension to the nonlinear intensity function was first introduced by Br\'{e}maud and Massouli\'{e} \cite{BremaudMassoulie1996AOP}, 
which is known as the nonlinear Hawkes process.
A multivariate  Hawkes process is  a counting process with the mutually-exciting property; it is one of the most popular  models to describe the interactions across the dimensions and also the dependence on the past events 
 (see e.g. Hawkes \cite{Hawkes1971Biometrika}, Hawkes and Oakes \cite{HawkesOakes1974JAP}, Br\'{e}maud and Massouli\'{e} \cite{BremaudMassoulie1996AOP},  Daley and Vere-Jones \cite{DaleyVere-Jones2003Books}, Zhu \cite{ZhuThesis} and the references therein).
The $N$-dimensional (multivariate nonlinear) Hawkes process is defined as $\left(Z_{t}^{1},\ldots,Z_{t}^{N}\right)$,
where $Z_{t}^{i}$, $1\leq i\leq N$, are simple point processes without common points,
with $Z_{t}^{i}$ admitting an $\mathcal{F}_{t}$-intensity:
\begin{equation}\label{dynamics-0}
\lambda_{t}^{i}:=\phi_{i}\left(\sum_{j=1}^{N}\int_{0}^{t-}h_{ij}(t-s)dZ_{s}^{j}\right)
\end{equation}
where $\phi_{i}(\cdot):\mathbb{R}^{+}\rightarrow\mathbb{R}^{+}$ is locally integrable, left continuous,
$h_{ij}(\cdot):\mathbb{R}^{+}\rightarrow\mathbb{R}^{+}$ and
we always assume that $\| h_{ij}\|_{L^{1}}=\int_{0}^{\infty}h_{ij}(t)dt<\infty$.
In the literature, $h_{ij}(\cdot)$ and $\phi_{i}(\cdot)$ are usually referred to
as exciting function (or sometimes kernel function) and rate function respectively. If  $\phi_{i}$ are all linear,  i.e.,  $ \phi_i(x)=\mu_i+x$ with $\mu_i\geq 0$, we obtain linear Hawkes processes where $\mu_i$ can be interpreted as a baseline Poisson intensity.  
By using the Poisson embeddings, see e.g. \cite{BremaudMassoulie1996AOP,DelattreFournierHoffmann2016AAP},
we can express the Hawkes process $\left(Z_{t}^{1},\ldots,Z_{t}^{N}\right)$
as the solution of a Poisson driven stochastic differential equation (SDE):
\begin{equation}\label{dynamics-sde}
Z^{i}_t=\int_0^t \int_0^\infty I_{\left\{z \leq \phi_{i}\left(\sum_{j=1}^N\int_0^{s-}h_{ij}(s-u)dZ_u^{j}\right)\right\}}
\pi^i(ds\,dz),
\qquad
1\leq i\leq N,
\end{equation}
where $\{\pi^i(ds\,dz), i\geq 1\}$ are a sequence of  i.i.d.  Poisson
random measures with common intensity measure $dsdz$ on
$[0,\infty) \times [0,\infty)$.

The large time (i.e. $t\rightarrow\infty$) limit theorems have been well studied for Hawkes processes in the literature.
In terms of linear Hawkes processes, Bacry et al. \cite{Bacry-etal-2013SPA} obtained functional law of large numbers and functional central limit theorem for the multivariate linear Hawkes process,
Bordenave and Torrisi \cite{BordenaveTorrisi2007SM} established large deviation principle for the univariate linear Hawkes process,
Gao and Zhu \cite{GaoZhu2021Bernoulli} obtained finer large deviations results utilizing the mod-$\phi$ convergence theory,
and Jaisson and Rosenbaum \cite{JaissonRosenbaum2015AAP,JaissonRosenbaum2016AAP} obtained limit theorems for the nearly unstable linear Hawkes processes.
In terms of nonlinear Hawkes processes, Zhu \cite{Zhu2013JAP} obtained a functional central limit theorem for the univariate nonlinear Hawkes process,
and the large deviations are derived in \cite{Zhu2014AIHP,Zhu2015AAP}.
There have also been studies for limit theorems taking asymptotics other than the large-time limit.
When the initial intensity is large, limit theorems have been obtained for linear Markovian Hawkes process (see e.g. \cite{GaoZhu2018SPA,GaoZhu2018Bernoulli})
and  when the baseline intensity is large (see e.g. \cite{GaoZhu2018QS,LiPang2022SPA}). Torrisi \cite{Torrisi2016AAP,Torrisi2017AIHP} studied the quantitative Gaussian and Poisson approximations of the simple point processes with stochastic intensity, which includes the univariate nonlinear Hawkes process as a special case.

The first work on the mean-field limit for high-dimensional Hawkes processes \eqref{dynamics-sde} appeared in Delattre et al. \cite{DelattreFournierHoffmann2016AAP}. 
They considered the (multivariate nonlinear) Hawkes process
$\left(Z^{N,1}_t,\cdots,Z^{N,N}_t\right)_{t\geq 0}$, defined by
\begin{equation} \label{N-dim-Hawkes-process-eq}
Z_t^{N,i}=\int_0^t \int_0^\infty I_{\left\{z \leq \phi \left(N^{-1}\sum_{j=1}^N\int_0^{s-}h(s-u)dZ_u^{N,j}\right)\right\}}
\pi^i(ds\,dz),\qquad i=1,\cdots,N,
\end{equation}
i.e.,  $h_{ij}=\frac{1}{N}h$
and $\phi_{i}=\phi$ in \eqref{dynamics-sde},  where   $\phi:\mathbb R \mapsto [0,\infty)$ is Lipschitz
and  $h:[0,\infty)\mapsto \mathbb R$ is a locally square integrable function.  The \textit{mean field} of the Hawkes processes is defined by
\begin{equation}\label{empirical-measure-def-eq}
L^N(t,dx):=\frac{1}{N}\sum_{i=1}^N \delta_{Z_t^{N,i}}(dx),\qquad t\geq 0,
\end{equation}
which is a $M_1(\mathbb N)$-valued stochastic process, where $M_1(\mathbb N)$  denotes  the space of probability measures on $\mathbb N:=\{0,1,2,\ldots\}$. 
Delattre et al. \cite{DelattreFournierHoffmann2016AAP} showed that
$L^N$ converges to its mean-field limit, i.e.
\begin{equation}\label{empirical-measure-MF-eq}
L^N\to \mathcal L \mbox { weakly in } D\left([0,\infty),M_1(\mathbb N)\right),
\end{equation}
where $\mathcal L_t(dx)$  is the distribution of $\widetilde Z_t$, which is the mean-field limit process that is defined by
\begin{equation}\label{Hawkes-mean-eq}
\widetilde Z_t = \int_0^t\int_0^\infty I_{ \left\{z \leq \phi\left(\int_0^{s}
h(s-u)d \mathbb E(\widetilde Z_u)\right)\right\}}\pi(ds\,dz),\qquad t\geq 0,
\end{equation}
where $\pi(ds \, dz)$ is a Poisson random measure on $[0,\infty)\times [0,\infty)$ with intensity measure $ds dz$.
Note that as observed in \cite{DelattreFournierHoffmann2016AAP}, $\widetilde{Z}_{t}$ given in \eqref{Hawkes-mean-eq}
is an inhomogeneous Poisson process with intensity $m_{t}= \mathbb E(\widetilde Z_t)$ that is the unique non-decreasing locally bounded solution of the following equation (see e.g. Delattre et al. \cite{DelattreFournierHoffmann2016AAP}):
\begin{equation} \label{mean-mean-eq}
m_t = \int_0^t \phi\left( \int_0^s h(s-u)dm_u\right)ds, \qquad t\geq 0.
\end{equation}
Since their seminal work, mean-field limits for Hawkes processes  
have attracted a lot of attention in the recent literature (e.g., see Agathe-Nerine \cite{Agathe-Nerine-2022SPA}, Borovykh et al. \cite{Borovykh-etal2018IISE},  Chevallier \cite{Chevallier2017SPA,Chevallier2018JSP}, Chevallier et al.\cite{Chevallier-etal2020SPA}, Delattre and Fournier \cite{DelattreFournier2016EJS},   Dittlevsen and L\"{o}cherbach \cite{Ditlevsen-etal2017SPA}, Duval et al.\cite{Duval-etal2022SPA}, 
Erny et al. \cite{Erny-etal2022Bernoulli}, L\"{o}cherbach \cite{Locherbach2019JTP}, Pfaffelhuber et al. \cite{Pfaffelhuber-etal2022SPA}, Raad et al. \cite{Raad-etal2020AIHP}, Schmutz \cite{Schmutz2022SPA}).
Fluctuations
of the \textit{mean process}:  
\begin{equation} \label{Hawkes-mean-process-eq-0}
\overline{Z}^{N}_t=\frac{1}{N}\sum_{i=1}^N Z_t^{N,i},\qquad t\geq 0,
\end{equation}
are studied by Gao and Zhu \cite{GaoFZhu2018SPA} and  Heesen and Stannat \cite{HeesenStannat2021SPA}.
Gao and Zhu \cite{GaoFZhu2018SPA} also studied the large deviations and moderate deviations for the mean process \eqref{Hawkes-mean-process-eq-0}.
In terms of the mean field \eqref{empirical-measure-MF-eq}, 
Gao and Zhu \cite{GaoFZhuMF-LDP2023} was the first one that obtained a large deviation principle.
In this paper, 
  we follow the same model setup as the first mean-field limit paper in the literature, i.e. \cite{DelattreFournierHoffmann2016AAP}, and we study the fluctuations of the  mean fields, 
  which is a central limit theorem that characterizes the fluctuations around the mean-field limit (Theorem~\ref{CLT-thm}) and the large deviations associated with the fluctuations, i.e.,  moderate deviations (Theorem~\ref{MDP-thm}), and this moderate deviation principle fills in the gap between the central limit theorem that is established in this paper (Theorem~\ref{CLT-thm}) and the large deviation principle obtained in \cite{GaoFZhuMF-LDP2023}.

This paper is organized as follows. In Section~\ref{sec:main},
we introduce the main results of the paper.  In Section~\ref{sec:exponential},  we provide some estimates for exponential moments of  Hawkes processes.  The fluctuation theorem  is shown in Section~\ref{CLTMeanFieldsSection} and the  moderate deviation principle is established in Section~\ref{MDPMeanFieldsSection}.  
Finally, in Section~\ref{sec:cor}, we show that our results (Theorem~\ref{CLT-thm} and Theorem~\ref{MDP-thm}) can recover 
the fluctuations and moderate deviations for the mean process obtained in \cite{GaoFZhu2018SPA}.

\section{Main Results}\label{sec:main}

\subsection{Preliminaries}\label{sec:pre}

In order to state our main results, let us first introduce the assumptions  and provide some preliminaries (cf.  \cite{DelattreFournierHoffmann2016AAP} and \cite{GaoFZhuMF-LDP2023}).
 

For each $N\geq 1$, we let $h_{ij}=\frac{1}{N}h$
and $\phi_{i}=\phi$ in \eqref{dynamics-sde} throughout the rest of the paper.
%
Moreover, we assume the following.

\begin{itemize}
\item[(A.1).]
$h(\cdot):[0,\infty)\rightarrow[0,\infty)$ is 
locally bounded,  i.e.,  for any $t>0$,
$$
\|h\|_0^t:=\sup_{s\in[0,t]}|h(s)|<\infty.
$$
Furthermore,  $h$ is differentiable and $|h'|$ is 
locally bounded.

\item[(A.2).]
$\phi(\cdot):[0,\infty)\rightarrow(0,\infty)$
is $\alpha$-Lipschitz for some $0<\alpha<\infty$, $T>0$ and 
$$
\alpha \| h\|_{L^{1}[0,T]} <1,
$$
where $\| h\|_{L^{1}[0,t]}:=\int_0^t|h(s)|ds$.

\item[(A.3).]
 $\phi'$ is uniformly continuous on $[0,\infty)$.
\end{itemize}

Next, let us introduce some technical backgrounds and notations.
\subsubsection*{The path space $D([0,T], \mathbb H^{-1}(\mathbb N))$}

Set
$$
C^{lip}(\mathbb N ):=\left\{\psi: \mathbb N \to\mathbb R;~\|\psi\|_{lip}<\infty\right\},
$$
where 
$$
\|\psi\|_{lip}:=|\psi(0)|+\sup_{x\ne y}\frac{|\psi(x)-\psi(y)|}{|x-y|}.
$$
Then $(C^{lip}(\mathbb N ),\|\cdot\|_{lip})$ is a Banach space. Let $\mathbb H^{-1}(\mathbb N):=C^{lip}(\mathbb N )^*$ be the topological dual space of  $(C^{lip}(\mathbb N ),\|\cdot\|_{lip})$, i.e.,  $\mathbb H^{-1}(\mathbb N):=\{\mu: C^{lip}(\mathbb N )\to \mathbb R  \mbox{ is  continuous linear mapping}\}$.
For any  $\mu\in \mathbb H^{-1}(\mathbb N)$, set
$$
\mu(x):=\mu\left(I_{\{x\}}\right),~~x\in \mathbb N,~~~\mbox{ and } \<\mu,\psi\>=\mu(\psi), ~~\psi\in(C^{lip}(\mathbb N ),\|\cdot\|_{lip}).
$$
Let $C_0(\mathbb N)$ be the space of bounded  continuous functions on $\mathbb N$ which  are equal  to $0$ at  $\infty$.   Denote by $\|\psi\|:=\sup_{x\in\mathbb N}|\psi(x)|$.  Note that
$C_0(\mathbb N)\subset C^{lip}(\mathbb N )$. Let $C_0(\mathbb N)^*$ be the dual space of $C_0(\mathbb N)$. Then we have $C_0(\mathbb N)^*\supset C^{lip}(\mathbb N )^*=\mathbb H^{-1}(\mathbb N)$ and by the Riesz representation theorem,  each $\mu \in\mathbb H^{-1}(\mathbb N)$ is a finite signed measure on $\mathbb N$, and
$$
 \<\mu,\psi\>=\sum_{x\in \mathbb N} \mu(x)\psi(x), ~~\psi\in(C^{lip}(\mathbb N ),\|\cdot\|_{lip}).
$$
We endow $\mathbb H^{-1}(\mathbb N)$  with the weak*  topology such that for $\mu_n,\mu\in \mathbb H^{-1}(\mathbb N)$,
$$
\mu_n\to\mu   \mbox{ if and only if } \<\mu_n,\psi\>\to \<\mu,\psi\>  \mbox{ for any } \psi\in C^{lip}(\mathbb N ).
$$
Let  $D([0,T], \mathbb H^{-1}(\mathbb N))$ denote the space of $ \mathbb H^{-1}(\mathbb N)$- valued c\`{a}dl\`{a}g functions on $[0,T] $ equipped with the topology of uniform convergence: 
$$
\mu^n\to \mu \mbox{  if and only if } 
\sup_{t\in[0,T]}|\<\mu_t^n-\mu_t,\psi\>|\to 0  \mbox{ for any } \psi\in C^{lip}(\mathbb N ).
$$
Define
$$
C^{lip}([0,T]\times\mathbb N ):=\left\{
\psi:[0,T]\times \mathbb N\to\mathbb R \text{  continuous with  }\sup_{t\in[0,T]}\|\psi(t)\|_{lip}<\infty \right\},
$$
where $\psi(t)(x):=\psi(t,x),~x\in\mathbb N$,
and
$$
C^{1,lip}([0,T]\times\mathbb N ):=\left\{\psi\in C^{lip}([0,T]\times\mathbb N );~\partial_t\psi \in C^{lip}([0,T]\times\mathbb N )
\right\}.
$$

Define $\nabla \psi(s,x):=\psi(s,x+1)-\psi(s,x)$ for any $x=0,1,2,\ldots$, and for any $s\in[0,T]$, denote by $\nabla \psi(s):\mathbb{N}\ni x\rightarrow\nabla\psi(s,x)$. We also denote 
$$
\|\psi\|:=\sup_{t\in[0,T]}\sup_{x\in\mathbb{N}}|\psi(t,x)|,
$$ 
for any $\psi\in C([0,T]\times\mathbb N )$.    

 
\subsubsection*{Mean-field densities}

We define  the mean-field density  and  the  centered mean-field density of the Hawkes processes as follows:
\begin{equation}\label{empirical-measure-def-eq-1}
L_t^N(x):=\frac{1}{N}\sum_{i=1}^N \delta_{Z_t^{N,i}}(\{x\}),~~ t\geq 0, ~x\in\mathbb N,
\end{equation}
and
\begin{equation}\label{central-mu-clt-def-1}
 \widehat{L}_t^N(x) := \frac{1}{\sqrt{N}}\sum_{i=1}^N
\left( \delta_{Z_t^{N,i}}(\{x\})-\mathcal L_t(x)\right),~~ t\geq 0, ~x\in\mathbb N,
\end{equation}
where $\mathcal L_t(x):=\mathcal L_t(\{x\})$, and $\mathcal L_t$  is the distribution of $\widetilde Z_t$ which is defined by
\eqref{Hawkes-mean-eq}.

We also consider the following rescaled centered mean-field density
\begin{equation}\label{central-mu-mdp-def-1}
 \widetilde{L}_t^N(x) := \frac{1}{\sqrt{N}a(N)}\sum_{i=1}^N
\left( \delta_{Z_t^{N,i}}(\{x\})-\mathcal L_t(x)\right),~~ t\geq 0, ~x\in\mathbb N,
\end{equation}
where $a(N),N\geq 1$ is a positive sequence with
\begin{equation}\label{mdp-speed-seq}
\lim_{N\to \infty} a(N)  = \infty,\qquad
 \lim_{N\to \infty} \frac{a(N)}{\sqrt{N}} =0.
\end{equation}

\subsubsection*{Large deviation principle}
A sequence $\{P_{n}, n\in\mathbb{N}\}$ of probability measures on a topological space $X$
is said to satisfy the large deviation principle (LDP) with the rate function $I:X\rightarrow\mathbb{R}\cup\{\infty\}$ 
and the speed $b_{n}$
if $I$ is non-negative,
lower semicontinuous and $b_{n}$ is a positive sequence with $\lim_{n\rightarrow\infty}b_{n}=\infty$
and for any measurable set $A$, with $A^{o}$ denoting the interior of $A$ and $\overline{A}$ being its closure,
\begin{equation}\label{eqn:LDP}
-\inf_{x\in A^{o}}I(x)\leq\liminf_{n\rightarrow\infty}\frac{1}{b_{n}}\log P_{n}(A)
\leq\limsup_{n\rightarrow\infty}\frac{1}{b_{n}}\log P_{n}(A)\leq-\inf_{x\in\overline{A}}I(x).
\end{equation}
We refer to Dembo and Zeitouni \cite{DemboZeitouni1998Book}
and Varadhan \cite{Varadhan1984Book}
for general background regarding the theory of large deviations and their applications.

\subsection{Main results}\label{sec:main:results}

Next, let us state the two main results in this paper. 
The first main result is a fluctuation theorem of the mean fields, 
which is a central limit theorem that characterizes the fluctuations
around the mean-field limit.

\begin{thm}[Fluctuation theorem of the mean fields]\label{CLT-thm}
Suppose (A.1), (A.2)  and   (A.3) hold.  Then
 $\left\{\widehat{L}^N,N\geq 1\right\}$ converges weakly   in  $D([0,T], \mathbb H^{-1}(\mathbb N)) $  to the Gaussian process $X_t(x)$ defined by
\begin{equation}\label{CLT-thm-eq}
\begin{aligned}
\<X_T,\varphi(T)\>&=  \int_0^T \<X_s,\partial_s\varphi(s)\>ds+
 \int_0^{T}\<X_s, \nabla \varphi(s)\>    \phi\left(\int_0^{s} h(s-u)d m_u\right)ds\\
&\qquad+ \int_0^{T}\<\mathcal L_s,\nabla \varphi(s)\> \phi'\left(\int_0^{s} h(s-u)d m_u\right) \int_0^{s} h(s-u)d\<X_u,\ell\>ds\\
&\qquad\qquad+ \int_0^T \sqrt{  \phi\left(\int_0^{s}
h(s-u)d m_u\right) }\left\<\sqrt{\mathcal L_s}\nabla\varphi(s), d  B_s\right\>,
\end{aligned}
\end{equation}
for any $\varphi \in C^{1,lip}([0,T]\times\mathbb N)$, where $\ell(x)=x$ for any $x\in\mathbb N$,  $B_t=\{\{B_t(x),t\geq 0\},x\in\mathbb N\}$ is a sequence of independent standard Brownian motions, and $\left\<\sqrt{\mathcal L_s}\nabla\varphi(s), d  B_s\right\>:=\sum_{x\in\mathbb N}  \sqrt{\mathcal L_s(x)}\nabla\varphi(s,x)d  B_s(x)$.
\end{thm}

\begin{rmk} 
Note that the conclusion of Theorem~\ref{CLT-thm} holds in $D([0,t], \mathbb H^{-1}(\mathbb N)) $ for any $t\in [0,T]$. In particular,  the equation~\eqref{CLT-thm-eq} holds for any $0\leq t\leq T$. 
 \end{rmk}

We can apply continuous mapping theorem and Theorem~\ref{CLT-thm}
to recover the fluctuation theorem for mean process of $\overline{Z}^{N}_t$ (defined in \eqref{Hawkes-mean-process-eq-0})
obtained in Gao and Zhu \cite{GaoFZhu2018SPA} (see Theorem~23 and Remark~7 in \cite{GaoFZhu2018SPA}).
The proof of the following corollary will be postponed to Section~\ref{sec:cor}.

\begin{cor}[Fluctuation theorem of the mean processes]\label{CLT-thm-mean-p}
Suppose (A.1), (A.2)  and   (A.3) hold.  Then
 $\left\{\sqrt{N}\left(\overline{Z}^{N}_t-m_{t}\right),0\leq t\leq T\right\}$
 converges weakly   in  $D([0,T], \mathbb R) $  to the Gaussian process $\overline{X}_t$ defined by
\begin{equation}\label{CLT-thm-mean-p-eq}
\begin{aligned}
\overline{X}_t&=  \int_0^{t}  \phi'\left(\int_0^{s} h(s-u)d m_u\right) \int_0^{s} h(s-u)d\overline{X}_uds\\
&\qquad\qquad+ \int_0^t \sqrt{  \phi\left(\int_0^{s}
h(s-u)d m_u\right) } dW_s,\qquad t\in[0,T],
\end{aligned}
\end{equation}
where   $W_t $ is a  standard Brownian motion.
\end{cor}

Next, let us state the second main result of the paper, which is a moderate deviation principle of the mean fields.

\begin{thm}[Moderate deviation principle of the mean fields]\label{MDP-thm}
 Suppose (A.1), (A.2)  and   (A.3) hold.  
For  any $\mu\in D([0,T], \mathbb H^{-1}(\mathbb N)) $  and $\varphi \in C^{1,lip}([0,T]\times\mathbb N)$, define
\begin{equation}\label{mdp-rate-function-eq-1}
\begin{aligned}
{\Upsilon}_\mu(\varphi):=&\< \mu_T,\varphi(T)\>- \int_0^{T}\< \mu_{t},\partial_t \varphi(t)\>   dt-\int_0^{T}\< \mu_t, \nabla \varphi(t)\>   \phi\left(\int_0^{t} h(t-s)d m_s\right)  dt\\
&-\int_0^{T}\< \mathcal L_t, \nabla \varphi(t)\> \phi'\left(\int_0^{t}
h(t-s)d m_s\right)  \int_0^{t}
h(t-s)d \<\mu_s,\ell\> dt,
\end{aligned}
\end{equation}
where $\ell(x)=x$ for any $x\in\mathbb N$,  and also define
\begin{equation}\label{mdp-rate-function-eq-2}
\begin{aligned}
J_\mu(\varphi):=&{\Upsilon}_\mu(\varphi)
- \frac{1}{2}\int_0^T\left\< \mathcal L_t, (\nabla \varphi(t))^2\right\>  \phi\left(\int_0^{t}
h(t-s)d m_s\right) dt.
\end{aligned}
\end{equation}
Then $\mathbb{P}\left(\widetilde{L}^N\in\cdot\right)$ satisfies a large deviation principle on $D([0,T], \mathbb H^{-1}(\mathbb N))$
with the speed $a^2(N)$ and the rate function ${I}$ that is defined by
\begin{equation}\label{mdp-rate-function-eq-3}
{I}(\mu):=\sup \left\{J_\mu(\varphi);\varphi\in C^{1,lip} ([0,T]\times\mathbb N)\right\}\,.
\end{equation}
\end{thm}

We can apply contraction principle and Theorem~\ref{MDP-thm}
to recover the moderate deviations for mean process of $\overline{Z}^{N}_t$ (defined in \eqref{Hawkes-mean-process-eq-0})
obtained in Gao and Zhu \cite{GaoFZhu2018SPA} (see Theorem~25 in \cite{GaoFZhu2018SPA}).
The proof of the following corollary will be postponed to Section~\ref{sec:cor}.

\begin{cor}[Moderate deviations of the mean processes]\label{MDP-thm-mean-p}
Suppose (A.1), (A.2)  and   (A.3) hold.  Then 
$\mathbb{P}\left(\left\{\frac{\sqrt{N}(\overline{Z}_{t}^{N}-m_t)}{a(N)},0\leq t\leq T\right\}\in\cdot\right)$ satisfies a large deviation principle on $D([0,T], \mathbb R)$
with the speed $a^2(N)$ and the rate function ${J}$ defined by
\begin{equation}\label{plug:into-rate-J-eq-4}
J(\eta)=
\begin{cases}
\frac{1}{2}\int_{0}^{T}\frac{\left(\partial_t\eta_{t}-\phi'\left(\int_0^{t}
h(t-s)d m_s\right)  \int_0^{t}h(t-s)d \eta_{s}\right)^2}{\phi\left(\int_{0}^{t}h(t-s)dm_{s}\right)}dt& \text{if $\eta=\{\eta_{t},0\leq t\leq T\} \in\mathcal{AC}_{0}[0,T]$},
\\
+\infty  & \text{otherwise}.
\end{cases}
\end{equation}
where $\mathcal{AC}_{0}[0,T]$ denotes the space of functions
$f:[0,T]\rightarrow\mathbb{R}$ that are absolutely continuous
with $f(0)=0$.
\end{cor}



\section{Some Estimates for Exponential Moments}\label{sec:exponential}

In this section, we  establish some estimates for exponential moments
and these estimates for exponential moments will be used in the proof of the fluctuation theorem (Theorem~\ref{CLT-thm}) in Section~\ref{CLTMeanFieldsSection}
and the moderate deviation principle (Theorem~\ref{MDP-thm}) in Section~\ref{MDPMeanFieldsSection}.

First, in the following technical lemma, we provide estimates for exponential moments
for the mean process $\overline{Z}_{T}^{N}$ defined in \eqref{Hawkes-mean-process-eq-0}.

\begin{lem}\label{Basic-lem-0}
Suppose (A.1), (A.2)  and   (A.3) hold.  

(1).  There exists $\theta_0>0$ such that  for any $\theta\in [0,\theta_0]$,    and for all  $N\geq 1$,
\begin{equation}\label{Basic-lem-0-eq-1}
\mathbb{E}\left(e^{\theta N\overline{Z}_{T}^{N}}\right)
\leq e^{2 N\theta\phi(0)T/(1-\alpha\| h\|_{L^{1}[0,T]})}.
\end{equation}
In particular, there exists $\hat\theta_0>0$ such that  
\begin{equation}\label{Basic-lem-0-eq-2}
\sup_{N\geq 1}\mathbb{E}\left(e^{\hat\theta_0\overline{Z}_{T}^{N}}\right)<\infty.
\end{equation}

(2).   There exists $\hat{\theta}_0>0$ such that  for any $r>0$,  $\theta\in (0,\hat{\theta}_0)$,
\begin{equation}\label{Basic-lem-0-eq-3}
\begin{aligned}
&\mathbb P\left(A_0N\sup_{0\leq t\leq T}\left|\overline{Z}_{t}^{N}-m_{t}\right|\geq r\right)
\leq 
\begin{cases}
2\exp\left\{-\frac{r^2}{2 A_1N }\right\} & \text{if $r \leq  A_1A_0^2N\hat{\theta}_0$},
\\
2\exp\left\{-\frac{\hat{\theta}_0 r}{2A_0}\right\} & \text{if $r> A_1A_0^2N\hat{\theta}_0$},
\end{cases}
\end{aligned}
\end{equation}
where 
\begin{equation*}
A_0:= e^{-\|\phi\|_{lip}\left(h(0)+\| h'\|_{L^{1}[0,T]}\right)T},  
\qquad
A_1:=\frac{6 \phi(0)T(1+ 1/(1-\alpha\| h\|_{L^{1}[0,T]}))}{A_0^2}.  
\end{equation*}
Furthermore, for any  $\theta\in (0,\hat\theta_0/4]$,
\begin{equation}\label{Basic-lem-0-eq-4}
\begin{aligned}
&\mathbb E\left(\exp\left\{A_0 \theta N\sup_{0\leq t\leq T}\left|\overline{Z}_{t}^{N}-m_{t}\right|\right\}\right)\\
&\leq 2\theta  \sqrt{2\pi A_1N} \exp\left\{\frac{\theta ^2A_1N}{2 }\right\} + \frac{ 4\theta } {\hat{\theta}_0 -2\theta }\exp\left\{-\frac{(\hat{\theta}_0 -2\theta)A_1\hat{\theta}_0N }{2} \right\} +1. 
\end{aligned}
\end{equation}

In particular,
\begin{equation}\label{Basic-lem-0-eq-5}
\limsup_{L\to\infty}\limsup_{N\to\infty}\frac{1}{a^2(N)} \log \mathbb P\left(\frac{\sqrt{N}}{a(N)}\sup_{0\leq t\leq T}\left|\overline{Z}_{t}^{N}-m_{t}\right|\geq  L\right)=-\infty,
\end{equation}
and
\begin{equation}\label{Basic-lem-0-eq-6}
\sup_{N\geq 1}\mathbb E\left(\exp\left\{\frac{A_0\hat{\theta}_0}{4} \sqrt{N}\sup_{0\leq t\leq T}\left|\overline{Z}_{t}^{N}-m_{t}\right|\right\}\right)
<\infty.
\end{equation}
\end{lem}

\begin{proof} (1).  It is obvious that
\begin{align*}
M_{t}^{N} &:=N\left(\overline{Z}_{t}^{N}- \int_{0}^{t}\phi\left(\int_{0}^{s}h(s-u)d\overline{Z}_{u}^{N}\right)ds\right)
\\
&=\sum_{i=1}^N\int_{0}^{t}\int_{0}^{\infty}I_{\{z\leq\phi(\int_{0}^{s-}h(s-u)d\overline{Z}_{u}^{N})\}}
(\pi^{i}(ds dz)- dsdz)
\end{align*}
is a martingale with the predictable quadratic variation:
\begin{equation}\label{N-dim-matingale-variation}
\left\langle M^{N}\right\rangle_{t}
=N\int_{0}^{t}\phi\left(\int_{0}^{s}h(s-u)d\overline{Z}_{u}^{N}\right)ds\leq N\left(\phi(0)t+\alpha\| h\|_{L^{1}[0,T]}  \overline{Z}_{t}^{N}\right).
\end{equation}
Since  $\sum _{i=1}^N \pi^{i}(ds dz)$ is a Poisson point process,  we have that  the jumps  of $M^N$  satisfy:   
$$
\left| \Delta M^{N}\right|\leq 1. 
$$  
where $\Delta M_t^{N}:=M_t^N-M_{t-}^N $. By Lemma 26.19 in Kallenberg \cite{Kallenberg2002Book} (or Proposition~2 in  \cite{ShorackWellnerBook1986} Appendix B),  for any $\theta\geq 0$,
$$
\mathcal E_t^{M^N,\theta}:=\exp\left\{\theta  N\overline{Z}_{t}^{N}
- (e^{\theta}-1)\left\<M^N\right\>_t\right\}=\exp\left\{\theta M_t^N - g(\theta)\left\<M^N\right\>_t\right\}
$$
is a positive supermartingale, where 
\begin{equation}\label{defn:g}
g(\theta):= e^\theta-\theta-1,\qquad\text{for $\theta\geq 0$}.
\end{equation}  
Thus,
$$
\begin{aligned}
1&\geq\mathbb{E}\left(\mathcal E_T^{M^N,\theta}\right)\geq\mathbb{E}\left(e^{\theta N \overline{Z}_{T}^{N}
-(e^{\theta}-1)N\phi(0)T
-(e^{\theta}-1)N\alpha\| h\|_{L^{1}[0,T]}  \overline{Z}_{T}^{N}}\right).
\end{aligned}
$$
Since we assumed that $\alpha\| h\|_{L^{1}[0,T]}<1$, 
 we can choose   $\theta_0>0$ such that  $\theta-(e^{\theta}-1)\alpha\| h\|_{L^{1}[0,T]}\geq 0$ for any $\theta\in [0,\theta_0]$.
It follows that for any $\theta\in [0,\theta_0]$,  $N\geq 1$,
$$
\mathbb{E}\left(e^{(\theta-(e^{\theta}-1)\alpha\| h\|_{L^{1}[0,T]})N\overline{Z}_{T}^{N}}\right)
\leq e^{(e^{\theta}-1)N\phi(0)T}.
$$
Note that $\lim_{\theta \to 0+}\frac{\theta-(e^{\theta}-1) \alpha\| h\|_{L^{1}[0,T]}}{\theta}=1-\alpha\| h\|_{L^{1}[0,T]}$ and  $\lim_{\theta \to 0+}\frac{e^{\theta}-1}{\theta}=1$. 
Now, by the H\"older inequality,  we obtain that
there exists $\theta_0>0$ such that  for any $\theta\in [0,\theta_0]$,  and for all  $N\geq 1$,
\begin{equation*}
\mathbb{E}\left(e^{\theta N\overline{Z}_{T}^{N}}\right)
\leq e^{2 N\theta\phi(0)T/(1-\alpha\| h\|_{L^{1}[0,T]})}.
\end{equation*}
In particular, for any $N\geq 1$, $\theta_{0}/N\in[0,\theta_{0}]$ such that
\begin{equation*}
\mathbb{E}\left(e^{\theta_0\overline{Z}_{T}^{N}}\right)
=\mathbb{E}\left(e^{(\theta_0/N)N\overline{Z}_{T}^{N}}\right)
\leq e^{2 N(\theta_{0}/N)\phi(0)T/(1-\alpha\| h\|_{L^{1}[0,T]})}
=e^{2\theta_{0}\phi(0)T/(1-\alpha\| h\|_{L^{1}[0,T]})},
\end{equation*}
which implies that $\sup_{N\geq 1}\mathbb{E}\left(e^{\theta_{0}\overline{Z}_{T}^{N}}\right)<\infty$.

(2).  Note that
\begin{equation}\label{Basic-lem-0-eq-8}
\overline{Z}_{t}^{N}-m_{t}
=\frac{1}{N}M_{t}^{N}+\int_{0}^{t}\phi\left(\int_{0}^{s}h(s-u)d\overline{Z}_{u}^{N}\right)ds
-\int_{0}^{t}\phi\left(\int_{0}^{s}h(s-u)dm_{u}\right)ds.
\end{equation}
By the assumptions and integration by parts formula,
$$
\begin{aligned}
&\left|  \phi\left(\int_0^{t-} h(t-s)d \overline{Z}^{N}_s\right)-\phi\left(\int_0^{t} h(t-s)d m_s\right) \right|
\\
&\leq\|\phi\|_{lip}
\left|\int_0^{t-} h(t-s)d \overline{Z}^{N}_s
-\int_0^{t} h(t-s)d m_s\right|
\\
&=\|\phi\|_{lip}\left|h(0)\overline{Z}_{t-}^{N}+\int_{0}^{t}h'(t-s)\overline{Z}_{s}^{N}ds
-h(0)m_{t}-\int_{0}^{t}h'(t-s)m_{s}ds\right|
\\
&\leq\|\phi\|_{lip}\left(h(0)+\| h'\|_{L^{1}[0,T]}\right)\sup_{0\leq s\leq t}\left|\overline{Z}_{s}^{N}-m_{s}\right|.
\end{aligned}
$$
Thus
\begin{equation*}
\sup_{0\leq s\leq t}\left|\overline{Z}_{s}^{N}-m_{s}\right|
\leq
\frac{1}{N}\sup_{0\leq s\leq T}\left|M_{s}^{N}\right|
+\|\phi\|_{lip}\left(h(0)+\| h'\|_{L^{1}[0,T]}\right)
\int_{0}^{t}\sup_{0\leq u\leq s}\left|\overline{Z}_{u}^{N}-m_{u}\right|ds.
\end{equation*}
By the Gronwall inequality, we get
\begin{equation}\label{Basic-lem-0-eq-9}
\sqrt{N}\sup_{0\leq t\leq T}\left|\overline{Z}_{t}^{N}-m_{t}\right|
\leq
\frac{1}{\sqrt{N}}\sup_{0\leq t\leq T}\left|M_{t}^{N}\right|
e^{\|\phi\|_{lip}\left(h(0)+\| h'\|_{L^{1}[0,T]}\right)T}.
\end{equation}
Then by \eqref{Basic-lem-0-eq-1}, 
there exists some $\hat{\theta}_{0}>0$ such that for any $\theta\leq\hat{\theta}_{0}$,
\begin{equation}\label{ineq:mid:step}
\mathbb E\left( \exp\left\{\theta N \overline{Z}_{T}^{N}\right\}\right)\leq e^{2N\theta  \phi(0)T/(1-\alpha\| h\|_{L^{1}[0,T]})}.
\end{equation}
Moreover we take $\hat{\theta}_{0}>0$ to be sufficiently small such that
\begin{equation}\label{g:inequality}
g(2\theta)=e^{2\theta}-2\theta-1\leq 3\theta^{2},\qquad\text{for any $0\leq\theta\leq\hat{\theta}_{0}$},
\end{equation}
where $g$ was first defined in \eqref{defn:g}.
Since  $\mathbb{E}\left(\mathcal E_T^{\pm M^N,\theta}\right)\leq 1$ for any $\theta\geq 0$, we have that for any $\theta\in (0,\hat\theta_0)$,
\begin{align}
&\mathbb E\left( \exp\left\{\theta  \left(\pm M_{T}^{N}\right)\right\}\right)\nonumber\\
&=
\mathbb E\left( \exp\left\{\theta \left(\pm  M_{T}^{N}\right)- \frac{1}{2}g(2\theta )\left\<M^N\right\>_T+\frac{1}{2}g(2\theta )\left\<M^N\right\>_T \right\}\right)\nonumber \\
&\leq 
\left(\mathbb E\left( \mathcal E_T^{\pm M^N,2\theta }\right)\right)^{1/2} \left(\mathbb E\left( \exp\left\{g(2\theta )\left\<M^N\right\>_T \right\}\right)\right)^{1/2}\label{ineq:first}  \\
&\leq   \left(\mathbb E\left( \exp\left\{g(2\theta ) N\left(\phi(0)T+\alpha\| h\|_{L^{1}[0,T]}  \overline{Z}_{T}^{N}\right)\right\}\right)\right)^{1/2}\label{ineq:second}\\
&\leq  \left(\mathbb E\left( \exp\left\{3\theta^2 N \left(\phi(0)T+\alpha\| h\|_{L^{1}[0,T]}  \overline{Z}_{T}^{N}\right)\right\}\right)\right)^{1/2}\label{ineq:third} \\
&\leq  e^{ 3\theta^2 N \phi(0)T(1+ 1/(1-\alpha\| h\|_{L^{1}[0,T]}))}\label{ineq:fourth},
\end{align}
where we applied Cauchy-Schwarz inequality to obtain \eqref{ineq:first}, and we applied \eqref{N-dim-matingale-variation}
together with $\mathbb{E}\left(\mathcal E_T^{\pm M^N,2\theta}\right)\leq 1$ to obtain \eqref{ineq:second}.
Moreover, the inequality \eqref{ineq:third} was due to \eqref{g:inequality}
and finally the inequality \eqref{ineq:fourth} was due to $\alpha\Vert h\Vert_{L^{1}[0,T]}<1$ and the inequality \eqref{ineq:mid:step}.

By the maximal inequality for martingales, we have that  for any $r>0$,  $\theta\in (0,\hat\theta_0)$,
$$ 
\begin{aligned}
\mathbb P\left(\sup_{0\leq t\leq T}\left|M_{t}^{N}\right|\geq   r\right)
\leq &
\mathbb E\left( \exp\left\{\theta   M_{T}^{N}\right\}\right) e^{-\theta  r}+\mathbb E\left( \exp\left\{-\theta  M_{T}^{N}\right\}\right) e^{-\theta  r }\\
\leq & 2 e^{-\theta  r  + 3\theta^2 N \phi(0)T(1+ 1/(1-\alpha\| h\|_{L^{1}[0,T]}))}.
\end{aligned}
$$ 
Therefore
$$ 
\begin{aligned}
\mathbb P\left(\sup_{0\leq t\leq T}\left|M_{t}^{N}\right|\geq   r\right)
&\leq  2  e^{-\sup_{0\leq \theta\leq \hat{\theta}_0}\{\theta  r  - \theta^2 A_2N/2)\}}
\leq 
\begin{cases}
2\exp\left\{-\frac{r^2}{2 A_2N }\right\} & \text{if  $r \leq A_2N\hat{\theta}_0$},
\\
2\exp\left\{-\frac{\hat{\theta}_0 r}{2}\right\} & \text{if  $r> A_2N\hat{\theta}_0$},
\end{cases}
\end{aligned}
$$
where $A_2:=6  \phi(0)T(1+ 1/(1-\alpha\| h\|_{L^{1}[0,T]}))$, 
and thus \eqref{Basic-lem-0-eq-3} is valid.
In particular,  \eqref{Basic-lem-0-eq-5} holds.

Furthermore, for any  $\theta\in (0,\hat\theta_0/4)$,
$$ 
\begin{aligned}
&\mathbb E\left(\exp\left\{ \theta \sup_{0\leq t\leq T}\left|M_{t}^{N}\right|\right\}\right)-1\\
=& \int_{0}^ {\infty} \mathbb P\left( \sup_{0\leq t\leq T}\left|M_{t}^{N}\right| \geq x\right) \theta e^{\theta x}dx \\
\leq &  \int_{0}^ {A_2\hat{\theta}_0N} 2\exp\left\{-\frac{x^2}{2A_2N}\right\}\theta e^{\theta x}dx +  \int_{ A_2\hat{\theta}_0 N} ^\infty 2\exp\left\{-\frac{\hat{\theta}_0 x}{2}\right\} \theta e^{\theta x}dx \\ 
\leq &  \exp\left\{\frac{\theta ^2A_2N}{2}\right\}  \int_{0}^ {\infty} 2\theta \exp\left\{-\frac{(x-A_2\theta N)^2}{2A_2N}\right\}dx
+ \frac{ 4\theta } {\hat{\theta}_0 -2\theta }\exp\left\{-\frac{(\hat{\theta}_0 -2\theta)A_2\hat{\theta}_0 N}{2}  \right\}  \\ 
\leq &  2\theta  \sqrt{2\pi A_2N} \exp\left\{\frac{\theta ^2A_2N}{2 }\right\} + \frac{ 4\theta } {\hat{\theta}_0 -2\theta }\exp\left\{-\frac{(\hat{\theta}_0 -2\theta)A_2\hat{\theta}_0N }{2} \right\}.  
\end{aligned}
$$  
That is,  for any  $\theta\in (0,\hat\theta_0/4)$,
\begin{equation}\label{Basic-lem-0-eq-10}
\begin{aligned}
&\mathbb E\left(\exp\left\{ \theta \sup_{0\leq t\leq T}\left|M_{t}^{N}\right|\right\}\right)\\
\leq &2\theta  \sqrt{2\pi A_2N} \exp\left\{\frac{\theta ^2A_2N}{2 }\right\} + \frac{ 4\theta } {\hat{\theta}_0 -2\theta }\exp\left\{-\frac{(\hat{\theta}_0 -2\theta)A_2\hat{\theta}_0N }{2} \right\}+1.
\end{aligned}
\end{equation}
Thus \eqref{Basic-lem-0-eq-4} is valid.
Specifically, we take $\theta=\hat\theta_0/(4\sqrt{N})$ in   \eqref{Basic-lem-0-eq-4}  to obtain \eqref{Basic-lem-0-eq-6}.
This completes the proof.
\end{proof}

Next, in the following technical lemma, we provide estimates for exponential moments
for $\sup_{0\leq t\leq T}\left|\left\<\widehat{L}^N_t,\varphi\right\>\right|$
for any test function $\varphi\in C^{lip}(\mathbb N)$, where $\widehat{L}^N_t$ is the centered mean-field density defined in \eqref{central-mu-clt-def-1}.

\begin{lem}\label{Basic-lem-1}
Suppose (A.1), (A.2)  and   (A.3) hold.  
Then there exist positive constants $\widetilde{C}_0,  \widetilde{C}_1,  \widetilde{C}_2, \widetilde{\theta}_0$ such that for any $\theta\in [0,\widetilde{\theta}_0]$,  $N\geq 1$, $\varphi\in C^{lip}(\mathbb N)$,
\begin{equation}\label{Basic-lem-1-eq-1}
\mathbb{E}\left(\exp\left\{ \theta  \sqrt{N}\sup_{0\leq t\leq T}\left|\left\<\widehat{L}^N_t,\varphi\right\>\right|\right\}\right)\leq   \widetilde{C}_0 (\|\varphi\|_{lip} \theta \sqrt{N}+1)e^{ \widetilde{C}_1\|\varphi\|_{lip}^2 \theta^2 N} + \widetilde{C}_2.
\end{equation}

In particular,
\begin{equation}\label{Basic-lem-1-eq-2}
\limsup_{L\to\infty}\limsup_{N\to\infty}\frac{1}{a^2(N)} \log \mathbb P\left(\frac{1}{a(N)}\sup_{0\leq t\leq T}\left|\left\<\widehat{L}^N_t,\varphi\right\>\right|\geq  L\right)=-\infty,
\end{equation}
and
\begin{equation}\label{Basic-lem-1-eq-3}
\sup_{N\geq 1}\mathbb E\left(\exp\left\{\frac{\widetilde{\theta}_0}{2} \sup_{0\leq t\leq T}\left|\left\<\widehat{L}^N_t,\varphi\right\>\right|\right\}\right)
<\infty.
\end{equation}
\end{lem}

\begin{proof} 
 For every $i\geq1$,  define the following inhomogeneous Poisson process:
\begin{equation}\label{Basic-lem-1-eq-4}
\widetilde{Z}^i_t = \int_0^t\int_0^\infty I_{ \left\{z \leq \phi\left(\int_0^s
h(s-u)d m_u\right)\right\}}\pi^i(ds\,dz),\;\;\text{ }\;\;t\geq 0,
\end{equation}
where $m_u,u\geq 0$ is the unique solution to \eqref{mean-mean-eq}, and $\{\pi^i,i\geq 1\}$ are  independent Poisson point processes. 
Then 
$$
\mathbb E\left(\widetilde{Z}^i_t \right)= \int_0^t\phi\left(\int_0^s
h(s-u)d m_u\right)ds=m_t. 
$$
Therefore, $\left\{\widetilde{Z}^i_t, t\geq0\right\}$, $i\geq 1$, are independent copies of $\widetilde{Z}$,  where $\widetilde{Z}$ is a solution to \eqref{Hawkes-mean-eq}. 

Then for any $\varphi \in C^{lip}(\mathbb N)$, 
$$
\begin{aligned}
\left|\left\<\widehat{L}^N_t,\varphi\right\>\right|= & \frac{1}{\sqrt{N}}\left|\sum_{i=1}^N\left(\varphi\left(Z_t^{N,i}\right)-\mathbb E\left(\varphi\left(\widetilde{Z}_t\right)\right)\right)\right|\\
\leq & \frac{1}{\sqrt{N}}\left|\sum_{i=1}^N\left(\varphi\left(Z_t^{N,i}\right)-\varphi\left(\widetilde{Z}_t^i\right)\right)\right|+ \frac{1}{\sqrt{N}}\left|\sum_{i=1}^N\left(\varphi\left(\widetilde{Z}_t^i\right)-\mathbb E\left(\varphi\left(\widetilde{Z}_t^i\right)\right)\right)\right|\\
\leq & \frac{\|\varphi\|_{lip}}{\sqrt{N}} \sum_{i=1}^N\left|Z_t^{N,i}-\widetilde{Z}_t^i\right|+ \frac{1}{\sqrt{N}}\left|\sum_{i=1}^N\left(\varphi\left(\widetilde{Z}_t^i\right)-\mathbb E\left(\varphi\left(\widetilde{Z}_t^i\right)\right)\right)\right|\\
= & \frac{\|\varphi\|_{lip} }{\sqrt{N}} \left(Y_t^{N,1}+  \sum_{i=1}^N \int_0^t  \int_0^\infty\left| I_{ \left\{z \leq \phi\left(\int_0^s
h(s-u)d m_u\right)\right\}}- I_{\left\{z \leq \phi\left(\int_0^{s-}+ 
h(s-u)d \overline{Z}_u^N\right)\right\}}\right|dzds\right)\\
&+ \frac{1}{\sqrt{N}} \left|Y_t^{N,2}\right|.
\end{aligned}
$$
where   
\begin{equation}\label{Basic-lem-1-eq-5}
Y_t^{N,1}:= \int_0^t  \int_0^\infty \left| I_{ \left\{z \leq \phi\left(\int_0^s
h(s-u)d m_u\right)\right\}}-I_{ \left\{z \leq \phi\left(\int_0^{s-}
h(s-u)d \overline{Z}_u^N\right)\right\}} \right|\sum_{i=1}^N(\pi^i(dsdz)-dsdz),
\end{equation}
and
\begin{equation}\label{Basic-lem-1-eq-6}
Y_t^{N,2}:= \sum_{i=1}^N\left(\varphi(\widetilde{Z}_t^i)-\mathbb E\left(\varphi\left(\widetilde{Z}_t^i\right)\right)\right).
\end{equation}
Note that 
$$
\begin{aligned}
&\frac{\|\varphi\|_{lip}}{\sqrt{N}} \sum_{i=1}^N \int_0^t  \int_0^\infty\left| I_{ \left\{z \leq \phi\left(\int_0^s
h(s-u)d m_u\right)\right\}}- I_{\left\{z \leq \phi\left(\int_0^{s-}+ 
h(s-u)d \overline{Z}_u^N\right)\right\}}\right|dzds\\
&=
 \|\varphi\|_{lip}\sqrt{N} \int_0^t \left| \phi\left(\int_0^s
h(s-u)d m_u\right)- \phi\left(\int_0^sh(s-u)d \overline{Z}_u^N\right)\right| ds\\
&=\|\varphi\|_{lip} \|\phi\|_{lip}\sqrt{N}\int_{0}^{t}\left|h(0)\left(m_{s}-\overline{Z}_{s}^{N}\right)+\int_{0}^{s}h'(s-u)\left(m_{u}-\overline{Z}_{u}^{N}\right)du\right|ds
\\
&\leq  \|\varphi\|_{lip} \|\phi\|_{lip}\left(h(0)+\| h'\|_{L^{1}[0,T]}\right)T
\sqrt{N}\sup_{0\leq u\leq t}\left|\overline{Z}_{u}^{N}-m_{u}\right|.
\\
\end{aligned}
$$
Therefore,
\begin{equation}\label{Basic-lem-1-eq-7}
\begin{aligned}
\left|\left\<\widehat{L}^N_t,\varphi\right\>\right|
&\leq   \frac{\|\varphi\|_{lip} }{\sqrt{N}} \left|Y_t^{N,1}\right|+ \frac{1}{\sqrt{N}} \left|Y_t^{N,2}\right|
\\
&\qquad + \|\varphi\|_{lip} \|\phi\|_{lip}\left(h(0)+\| h'\|_{L^{1}[0,T]}\right)T
\sqrt{N}\sup_{0\leq u\leq t}\left|\overline{Z}_{u}^{N}-m_{u}\right|.
\end{aligned}
\end{equation}

Let us first estimate the exponential moment of $Y_t^{N,1}$.
Note that $Y_t^{N,1}$ is a martingale with $\left| \Delta Y_t^{N,1}\right|\leq 1$, and
$$
\begin{aligned}
\left\<Y^{N,1}\right\>_t
&=N \int_0^t \left| \phi\left(\int_0^s
h(s-u)d m_u\right)- \phi\left(\int_0^sh(s-u)d \overline{Z}_u^N\right)\right| ds
\\
&\leq\|\phi\|_{lip}\left(h(0)+\| h'\|_{L^{1}[0,T]}\right)T
N\sup_{0\leq u\leq t}\left|\overline{Z}_{u}^{N}-m_{u}\right|.
\end{aligned}
$$

By Lemma~26.19 in Kallenberg \cite{Kallenberg2002Book} (or Proposition~2 in  \cite{ShorackWellnerBook1986} Appendix B),  for any $\theta\geq 0$,
$$
\mathcal E_t^{Y^{N,1},\theta}= \exp\left\{\theta Y_t^{N,1} - g(\theta)\left\<Y^{N,1}\right\>_t\right\}
$$
is a positive supermartingale, where $g(\theta)$ is defined in \eqref{defn:g}.  
Note that $g(2\theta)/\theta^2\to 2$ as $\theta\to 0$, by \eqref{Basic-lem-0-eq-4}  and  following the proof of \eqref{Basic-lem-0-eq-10},   there exist positive constants $C_1, C_2, \theta_1$ such that for any $\theta\in [0,\theta_1]$,  $N\geq 1$,
\begin{equation}\label{Basic-lem-1-eq-7}
\begin{aligned}
&\mathbb{E}\left(e^{ \theta \sup_{0\leq t\leq T}\left|Y_t^{N,1}\right|}\right) \leq  C_1 \theta \sqrt{N} e^{C_2 \theta^2 N}+1.
\end{aligned}
\end{equation}

Next, we estimate the exponential moment of $Y_t^{N,2}$.
  For any $\varphi \in C^{lip}(\mathbb N)$, 
  by It\^{o}'s formula,
\begin{equation}\label{Basic-lem-1-eq-9}
\begin{aligned}
& \sum_{i=1}^N\left(\varphi\left(\widetilde{Z}_t^i\right)-\mathbb E\left(\varphi\left(\widetilde{Z}_t^i\right)\right)\right)\\
=&  \sum_{i=1}^N\int_0^{t} \left(\nabla\varphi\left(\widetilde{Z}_s^i\right)-\mathbb E\left(\nabla\varphi\left(\widetilde{Z}_s^i\right)\right)\right) \phi\left(\int_0^{s}
h(s-u)d m_u\right)ds +\widetilde{M}_t^{\varphi,N},
\end{aligned}
\end{equation}
where
\begin{equation}\label{Basic-lem-1-eq-10}
\begin{aligned}
\widetilde{M}_t^{\varphi,N}:= \sum_{i=1}^N\int_0^{t} \int_0^\infty \nabla  \varphi\left(\widetilde{Z}_{s-}^{i}\right)I_{ \left\{z \leq \phi\left(\int_0^{s-}
h(s-u)dm_u\right)\right\}}   ({\pi}^{i}(ds\,dz)-dsdz)
\end{aligned}
\end{equation}
is a martingale with  $\left|\Delta \widetilde{M}_t^{\varphi,N}\right|\leq \|\varphi\|_{lip}$  and the predictable quadratic variation
$\left\<\widetilde{M}^{\varphi,N}\right\>_t$ that satisfies
\begin{equation}\label{MNT-variation}
\begin{aligned}
\left|\left\<\widetilde{M}^{\varphi,N}\right\>_t\right|=&\left|\sum_{i=1}^N\int_0^{t} \left| \nabla  \varphi\left(\widetilde{Z}_{s}^{i}\right)\right|^2 \phi\left(\int_0^{s}
h(s-u)d m_u\right)  ds\right|\\
=&\left|\sum_{i=1}^N\int_0^{t}  \left|\nabla  \varphi\left(\widetilde{Z}_{s}^{i}\right)\right|^2 dm_{s}\right|\leq N \|\varphi\|_{lip}^2m_{t}.
\end{aligned}
\end{equation}
Then for any $\theta\geq 0$,
$$
\mathcal E_t^{\widetilde{M}_t^{\varphi,N},\theta/\|\varphi\|_{lip}}= \exp\left\{\frac{\theta}{\|\varphi\|_{lip}} \widetilde{M}_t^{\varphi,N} - g(\theta/ \|\varphi\|_{lip})\left\<\widetilde{M}^{\varphi,N}\right\>_t\right\}
$$
is a positive  supermartingale.  Thus,   
following again the proof of \eqref{Basic-lem-0-eq-10},  there exist positive constants $C_3, C_4, \theta_2$ such that for any $\theta\in [0,\theta_2]$,  $N\geq 1$,
\begin{equation}\label{Basic-lem-1-eq-11}
\begin{aligned}
&\mathbb{E}\left(e^{ \theta  \sup_{0\leq t\leq T}\left|\widetilde{M}_t^{\varphi,N}\right|}\right) \leq  C_3\|\varphi\|_{lip} \theta \sqrt{N} e^{C_4\|\varphi\|_{lip}^2 \theta^2 N}+1.
\end{aligned}
\end{equation}

Note that
$$
\begin{aligned}
&\sup_{0\leq t\leq T}\left| \sum_{i=1}^N\int_0^t\left(\nabla\varphi\left(\widetilde{Z}_s^i\right)-\mathbb E\left(\nabla\varphi\left(\widetilde{Z}_s^i\right)\right)\right) \phi\left(\int_0^{s}
h(s-u)d m_u\right)ds\right|\\
\leq &\left( \phi(0)+\alpha\|h\|_0^Tm_T  \right)\int_0^T\left| \sum_{i=1}^N\left(\nabla\varphi\left(\widetilde{Z}_s^i\right)-\mathbb E\left(\nabla\varphi\left(\widetilde{Z}_s^i\right)\right)\right) \right|ds,
\end{aligned}
$$
and  
\begin{equation*}
\left|\nabla\varphi\left(\widetilde{Z}_s^i\right)-\mathbb E\left(\nabla\varphi\left(\widetilde{Z}_s^i\right)\right)\right|\leq 2\|\varphi\|_{lip}. 
\end{equation*}
By the Hoeffding's inequality \cite{HoeffdingJASA1963}, we have
\begin{equation}\label{Basic-lem-1-eq-12}
\begin{aligned}
&\mathbb{E}\left(\exp\left\{ \theta  \sup_{0\leq t\leq T}\left| \sum_{i=1}^N\int_0^t\left(\nabla\varphi\left(\widetilde{Z}_s^i\right)-\mathbb E\left(\nabla\varphi\left(\widetilde{Z}_s^i\right)\right)\right) \phi\left(\int_0^{s}
h(s-u)d m_u\right)ds\right|\right\}\right)\\
\leq & \frac{1}{T}\int_0^T\mathbb{E}\left(\exp\left\{ \theta T \left( \phi(0)+\alpha\|h\|_0^Tm_T  \right)\left| \sum_{i=1}^N\left(\nabla\varphi\left(\widetilde{Z}_s^i\right)-\mathbb E\left(\nabla\varphi\left(\widetilde{Z}_s^i\right)\right)\right) \right|\right\}\right)ds\\
\leq &\exp\left\{2\theta^2 N\|\varphi\|_{lip}^2 T^2 \left( \phi(0)+\alpha\|h\|_0^Tm_T\right)^2 \right\}.
\end{aligned}
\end{equation}
 
Combining with \eqref{Basic-lem-0-eq-4}, \eqref{Basic-lem-1-eq-7},  \eqref{Basic-lem-1-eq-11}  and \eqref{Basic-lem-1-eq-12}, 
we conclude that there exist positive constants $\widetilde{C}_0,  \widetilde{C}_1, \widetilde{\theta}_0$ such that for any $\theta\in [0,\widetilde{\theta}_0]$,  $N\geq 1$, $\varphi\in C^{lip}(\mathbb N)$,
$$
\mathbb{E}\left(\exp\left\{ \theta  \sqrt{N}\sup_{0\leq t\leq T}\left|\left\<\widehat{L}^N_t,\varphi\right\>\right|\right\}\right)\leq   \widetilde{C}_0 \left(\|\varphi\|_{lip} \theta \sqrt{N}+1\right)e^{ \widetilde{C}_1\|\varphi\|_{lip}^2 \theta^2 N}.
$$
 In particular,
by applying the Chebyshev inequality,  we obtain \eqref{Basic-lem-1-eq-2}.  Taking $\theta=\widetilde{\theta}_0/(2\sqrt{N})$,   we obtain \eqref{Basic-lem-1-eq-3}.
The proof is complete.
\end{proof}

 \begin{rmk} 
Furthermore,  we can also give some estimates for the error  between  the multivariate nonlinear Hawkes process
$\left(Z^{N,1}_t,\cdots,Z^{N,N}_t\right)$ defined in \eqref{N-dim-Hawkes-process-eq}   and  the  $N$ independent inhomogeneous Poisson processes $\widetilde{Z}^{1}_t,\cdots, \widetilde{Z}^{N}_t$ that are defined in \eqref{Basic-lem-1-eq-4}.
In fact,  for any $1\leq i\leq N$,  we have
$$
\begin{aligned}
&\left|Z_t^{N,i}-\widetilde{Z}^i_t\right|
\\
&=\int_0^t  \int_0^\infty \left| I_{ \left\{z \leq \phi\left(\int_0^s
h(s-u)d m_u\right)\right\}}- I_{ \left\{z \leq \phi\left(\int_0^{s-}
h(s-u)d \overline{Z}_u^N\right)\right\}}\right|   (\pi^i(dsdz)-dsdz)
\\
&\qquad\qquad+\int_0^t  \left| \phi\left(\int_0^s
h(s-u)d m_u\right)- \phi\left(\int_0^s
h(s-u)d \overline{Z}_u^N\right) \right|  ds
\\
&\leq\int_0^t  \int_0^\infty \left| I_{ \left\{z \leq \phi\left(\int_0^s
h(s-u)d m_u\right)\right\}}- I_{ \left\{z \leq \phi\left(\int_0^{s-}
h(s-u)d \overline{Z}_u^N\right)\right\}}\right|   (\pi^i(dsdz)-dsdz)
\\
&\qquad\qquad+\|\phi\|_{lip}\int_0^t  \left|\int_0^s
h(s-u)d \left(m_u-\overline{Z}_u^N\right) \right|  ds
\\
&=\int_0^t  \int_0^\infty \left| I_{ \left\{z \leq \phi\left(\int_0^s
h(s-u)d m_u\right)\right\}}- I_{ \left\{z \leq \phi\left(\int_0^{s-}
h(s-u)d \overline{Z}_u^N\right)\right\}}\right|   (\pi^i(dsdz)-dsdz)
\\
&\qquad\qquad+\|\phi\|_{lip}\int_0^t  \left|
h(0)\left(m_s-\overline{Z}_s^N\right)+\int_{0}^{s}h'(s-u)\left(m_u-\overline{Z}_u^N\right)du\right|  ds
\\
&\leq \int_0^t  \int_0^\infty \left| I_{ \left\{z \leq \phi\left(\int_0^s
h(s-u)d m_u\right)\right\}}- I_{ \left\{z \leq \phi\left(\int_0^{s-}
h(s-u)d \overline{Z}_u^N\right)\right\}}\right|   (\pi^i(dsdz)-dsdz)
\\
&\qquad\qquad+ \|\phi\|_{lip}\left(h(0)+\| h'\|_{L^{1}[0,T]}\right)T
\sup_{0\leq u\leq t}\left|\overline{Z}_{u}^{N}-m_{u}\right|,
\end{aligned}
$$
where $\overline{Z}_{u}^{N}$ is the mean process defined in \eqref{Hawkes-mean-process-eq-0}.
Thus,  under (A.1), (A.2)  and   (A.3),  by Lemma \ref{Basic-lem-0} and  using the same  argument as  \eqref{Basic-lem-1-eq-7},  
 there exist positive constants $\widetilde{C}_0,  \widetilde{C}_1,  \widetilde{C}_2, \widetilde{\theta}_0$ such that for any $\theta\in [0,\widetilde{\theta}_0]$,  $N\geq 1$,  
\begin{equation}\label{rmk-Basic-lem-1-eq-1}
\max_{1\leq i\leq N}\mathbb{E}\left(\exp\left\{ \theta  N\sup_{0\leq t\leq T}\left|Z_t^{N,i}-\widetilde{Z}^i_t\right|\right\}\right)\leq   \widetilde{C}_0  \theta \sqrt{N} e^{ \widetilde{C}_1 \theta^2 N}+ \widetilde{C}_2.
\end{equation}

In particular,
\begin{equation}\label{rmk-Basic-lem-1-eq-2}
\limsup_{L\to\infty}\limsup_{N\to\infty}\frac{1}{a^2(N)} \log \max_{1\leq i\leq N}\mathbb P\left(\frac{\sqrt{N}}{a(N)}\sup_{0\leq t\leq T}\left|Z_t^{N,i}-\widetilde{Z}^i_t\right|\geq  L\right)=-\infty,
\end{equation}
and
\begin{equation}\label{rmk-Basic-lem-1-eq-3}
\sup_{N\geq 1}\max_{1\leq i\leq N}\mathbb E\left(\exp\left\{\frac{\widetilde{\theta}_0}{2} \sqrt{N}\sup_{0\leq t\leq T}\left|Z_t^{N,i}-\widetilde{Z}^i_t\right|\right\}\right)
<\infty.
\end{equation}
 \end{rmk}

\section{Fluctuations for the Mean-Field Limit}\label{CLTMeanFieldsSection}

In this section, we prove Theorem~\ref{CLT-thm}, which is a fluctuation theorem for the mean-field limit. 
First,  for any $\varphi \in C^{1,lip}([0,T]\times\mathbb N)$, by It\^{o}'s formula,
\begin{equation}\label{Martingale-eq-0}
\begin{aligned}
&\left\<\widehat{L}^N_T,\varphi(T)\right\>\\
&= \frac{1}{\sqrt{N}}\sum_{i=1}^N\int_0^{T} \partial_t\left(\varphi\left(t,Z_t^{N,i}\right)-\mathbb E\left(\varphi\left(t,\widetilde{Z}_t\right)\right)\right)dt\\
&\qquad+\frac{1}{\sqrt{N}}\sum_{i=1}^N\int_0^{T} \int_0^\infty \nabla  \varphi\left(t,Z_{t-}^{N,i}\right)I_{ \left\{z \leq \phi\left(\int_0^{t-}
h(t-s)d \overline{Z}^{N}_s\right)\right\}} {\pi}^{i}(dt\,dz)\\
&\qquad\qquad-\sqrt{N}\int_0^{T} \mathbb E \left(\nabla  \varphi\left(t,\widetilde{Z}_{t}\right)\right)\phi\left(\int_0^{t} h(t-s)d m_s\right) dt\\
&= \int_0^{T}\left\<\widehat{L}_{t}^N,\partial_t \varphi(t)\right\>   dt+ \int_0^{T}\left\<\widehat{L}_t^N, \nabla \varphi(t)\right\>  \phi\left(\int_0^{t} h(t-s)d m_s\right) dt \\
&\qquad +\int_0^{T}\left\<{L}_t^N, \nabla \varphi(t)\right\>   \sqrt{N} \left( \phi\left(\int_0^{t-} h(t-u)d \overline{Z}^{N}_u\right)-\phi\left(\int_0^{t} h(t-u)d m_u\right)\right) dt+M_T^{\varphi, N},
\end{aligned}
\end{equation}
where
\begin{equation}\label{Martingale-eq-1}
\begin{aligned}
M_t^{\varphi,N}:=\frac{1}{\sqrt{N}}\sum_{i=1}^N\int_0^{t} \int_0^\infty \nabla  \varphi\left(s,Z_{s-}^{N,i}\right)I_{ \left\{z \leq \phi\left(\int_0^{s-}
h(s-u)d \overline{Z}^{N}_u\right)\right\}}   ({\pi}^{i}(ds\,dz)-dsdz), ~~t\in[0,T],
\end{aligned}
\end{equation}
is a martingale with the predictable quadratic variation
\begin{equation}\label{Martingale-eq-2}
\begin{aligned}
\left\<M^{\varphi,N}\right\>_t:=&\frac{1}{N}\sum_{i=1}^N\int_0^{t}  \left(\nabla  \varphi\left(s,Z_s^{N,i}\right)\right)^2\phi\left(\int_0^{s}h(s-u)d \overline{Z}^{N}_u\right) ds\,.
\end{aligned}
\end{equation}

Therefore,  in order to show Theorem~\ref{CLT-thm},  we need to study convergence of the terms in equation~\eqref{Martingale-eq-0} which will be completed by the establishing a sequence of lemmas that are stated as follows.
First, in the following lemma, we show that the martingale term $M_t^{\varphi,N}$ in \eqref{Martingale-eq-0}, which is defined in \eqref{Martingale-eq-1}, converges
to a Gaussian process.

\begin{lem}\label{CLT-thm-lem-1}
Suppose (A.1), (A.2)  and   (A.3) hold.  Then as $N\rightarrow\infty$,
$$
\left\<M^{\varphi,N}\right\>_T{\to}  \int_0^{T} \left\<\mathcal L_t,( \nabla \varphi(t))^2\right\>  \phi\left(\int_0^{t}
h(t-s)d m_s\right)  dt ,
$$
and
$\left\{M_t^{\varphi,N},t\in [0,T]\right\}$ converges weakly to the Gaussian process:
$$
\int_0^T \sqrt{  \phi\left(\int_0^{t}
h(t-u)d m_u\right) } \left\<\sqrt{\mathcal L_t} \nabla \varphi(t), dB_t\right\>.
$$
\end{lem}

\begin{proof} 
We first prove that $M^{\varphi,N}$ is $C$-tight in $D([0,T],\mathbb R)$, i.e., tightness in $D([0,T],\mathbb R)$ and  the set of limit points  is included in  $C([0,T],\mathbb R)$.  Note that for some positive constant $C$, we have
\begin{equation}\label{M-jump-bound-A}
\sup_{t\in[0,T]}\left|\Delta M_t^{\varphi,N}\right|\leq \frac{C}{\sqrt{N}},
\end{equation}
where $\Delta M_t^{\varphi,N}:=M_t^{\varphi,N}-M_{t-}^{\varphi,N}$.
We only need to show that for any $\eta>0$,
\begin{equation}\label{CLT-thm-lem-1-tight-1}
\lim_{\delta\rightarrow0}\lim_{N\rightarrow\infty}\mathbb{P}\left(\sup_{0\leq s,t\leq T,|s-t|\leq \delta}\left|M_t^{\varphi,N}-M_s^{\varphi,N}\right|\geq\eta\right)=0.
\end{equation}

First,   by the Doob's maximal  inequality,
\begin{equation*}
\begin{aligned}
\mathbb{E}\left(\sup_{t\in[0,T]}\left|M_t^{\varphi,N}\right|^2\right)\leq &4\mathbb{E}\left(\left(M^N_T\right)^2\right)\\
\leq &4\mathbb{E}\left( \int_0^{T} \left\<L_t^N,( \nabla \varphi(t))^2\right\>  \phi\left(\int_0^{t}h(t-s)d \overline{Z}^{N}_s\right)  dt\right)\\
\leq &4 \mathbb{E}\left( \int_0^{T} \|\nabla \varphi\|^2  \phi\left(\int_0^{t}
h(t-s)d \overline{Z}^{N}_s\right)  dt\right)\\
\leq &4\|\nabla \varphi\|^2 T\left( \phi(0)+\alpha\|h\|_0^T\mathbb{E}\left[\overline{Z}^{N}_T\right] \right).
\end{aligned}
\end{equation*}
Therefore,   by \eqref{Basic-lem-0-eq-2},  
\begin{equation}
\sup_{N\geq 1}\mathbb{E}\left(\sup_{t\in[0,T]}\left|M_t^{\varphi,N}\right|^2\right)<\infty.
\end{equation}
Without loss of generality, we assume that $T/\delta\in\mathbb{N}$.  Then by 
the Doob's maximal inequality
\begin{equation}\label{CLT-thm-lem-1-tight-2}
\begin{aligned}
&\mathbb{P}\left(\sup_{0\leq s,t\leq T,|s-t|\leq \delta}\left|M_t^{\varphi,N}-M_s^{\varphi,N}\right|\geq\eta\right)\\
\leq & \sum_{n=1}^{T/\delta}\mathbb{P}\left(\sup_{0\leq t\leq \delta}\left|M^N_{(n-1)\delta+t}-M^N_{(n-1)\delta}\right|\geq\eta/2\right)\\
\leq & \frac{16}{\eta^4}\sum_{n=1}^{T/\delta}\mathbb{E}\left(\left|M^N_{n\delta}-M^N_{(n-1)\delta}\right|^4\right).
\end{aligned}
\end{equation}

By applying the Burkholder-Davis-Gundy inequality for  the Poisson stochastic integrals, we have
\begin{equation}\label{CLT-thm-lem-1-tight-3}
\begin{aligned}
&\sum_{n=1}^{T/\delta}\mathbb{E}\left(\left|M^N_{n\delta}-M^N_{(n-1)\delta}\right|^4\right)\\
&\leq  \sum_{n=1}^{T/\delta}\mathbb{E}\left(\left( \int_{(n-1)\delta}^{n\delta} \<L_t^N,( \nabla \varphi(t))^2\>  \phi\left(\int_0^{t}
h(t-s)d \overline{Z}^{N}_s\right)  dt\right)^2\right)\\
&\leq4 \|\nabla\varphi\|^2\delta T\mathbb{E}\left(\left( \sup_{t\in[0,T]} \phi\left(\int_0^{t}h(t-s)d \overline{Z}^{N}_s\right)\right)^2\right)\\
&\leq4\|\nabla\varphi\|^2\delta T\mathbb{E}\left[\left( \phi(0)+\alpha\|h\|_0^T\overline{Z}^{N}_T  \right)^2\right]\\
&\leq8\|\nabla\varphi\|^2\delta T\left( |\phi(0)|^2+\alpha^2\left(\|h\|_0^T\right)^2\mathbb{E}\left[\left(\overline{Z}^{N}_T \right)^2\right] \right) .
\end{aligned}
\end{equation}
Therefore, by \eqref{Basic-lem-0-eq-2},   \eqref{CLT-thm-lem-1-tight-2} and \eqref{CLT-thm-lem-1-tight-3},  we complete the proof of \eqref{CLT-thm-lem-1-tight-1}.
 
Finally, by \cite{DelattreFournierHoffmann2016AAP}, we know that
$L^N_t$ converges to its mean-field limit $\mathcal L_t$  almost surely as $N\rightarrow\infty$. Then, as $N\to\infty$,
$$
\left\<M^{\varphi,N}\right\>_t{\to}  \int_0^{t} \<\mathcal L_s,( \nabla \varphi(s))^2\>  \phi\left(\int_0^{s}
h(s-u)d m_u\right)  ds,
$$
and
$\left\{M_t^{\varphi,N},t\in [0,T]\right\}$ converges weakly to a Gaussian process  with  mean $0$ and variance
$$
\left\<\mathcal L_t,( \nabla \varphi(t))^2\right\>  \phi\left(\int_0^{t} h(t-s)d m_s\right).
$$
Note that $ \int_0^T \sqrt{  \phi\left(\int_0^{t}
h(t-u)d m_u\right) }\<\sqrt{\mathcal L_t}\nabla\varphi(t),d  B_t\>$ is a Gaussian process  with mean $0$ and variance
$$
\left\<\mathcal L_t,( \nabla \varphi(t))^2\right\>  \phi\left(\int_0^{t}
h(t-s)d m_s\right).
$$
This completes the proof of the lemma.
\end{proof}

Next, we provide an estimate that concerns the third term on the right hand side of \eqref{Martingale-eq-0},
and this estimate will be used later to show that $\left\{\widehat{L}_t^N,N\geq 1\right\}$ is tight (Lemma~\ref{CLT-thm-lem-4}).

\begin{lem}\label{CLT-thm-lem-2}
Suppose (A.1), (A.2)  and   (A.3) hold.  Then for any $\epsilon>0$,
$$
\begin{aligned}
\lim_{N\to\infty}&\mathbb P\bigg( \sup_{t\in[0,T]}\bigg| \left( \phi\left(\int_0^{t} h(t-s)d \overline{Z}^{N}_s\right)-\phi\left(\int_0^{t} h(t-s)d m_s\right)\right) \bigg|\geq \epsilon\bigg)=0.
\end{aligned}
$$
\end{lem}

\begin{proof}
Since for any $0\leq t\leq T$:
\begin{equation}\label{CLT-thm-lem-2-eq-0}
\begin{aligned}
&\bigg| \left( \phi\left(\int_0^{t} h(t-s)d \overline{Z}^{N}_s\right)-\phi\left(\int_0^{t} h(t-s)d m_s\right)\right) \bigg|
\\
&\leq\|\phi\|_{lip}\left(h(0)+\| h'\|_{L^{1}[0,T]}\right)\sup_{0\leq s\leq t}\left|\overline{Z}_{s}^{N}-m_{s}\right|.
\end{aligned}
\end{equation}
By  \eqref{Basic-lem-0-eq-6}, 
\begin{equation}\label{N:bar:Z:finite}
\sup_{N\in\mathbb{N}}N\mathbb{E}\left(\sup_{0\leq t\leq T}\left|\overline{Z}_{t}^{N}-m_{t}\right|^{2}\right)<\infty.
\end{equation}
Thus, by Chebychev's inequality,
\begin{align*}
&\mathbb P\bigg( \sup_{t\in[0,T]}\bigg| \left( \phi\left(\int_0^{t} h(t-s)d \overline{Z}^{N}_s\right)-\phi\left(\int_0^{t} h(t-s)d m_s\right)\right) \bigg|\geq \epsilon\bigg)
\\
&\leq
\mathbb{P}\left(\|\phi\|_{lip}\left(h(0)+\| h'\|_{L^{1}[0,T]}\right)
\sup_{0\leq t\leq T}\left|\overline{Z}_{t}^{N}-m_{t}\right|
\geq\epsilon\right)
\\
&\leq\frac{\|\phi\|_{lip}^{2}\left(h(0)+\| h'\|_{L^{1}[0,T]}\right)^{2}
\sup_{N\in\mathbb{N}}N\mathbb{E}\left[\sup_{0\leq t\leq T}\left|\overline{Z}_{t}^{N}-m_{t}\right|^{2}\right]}{N\epsilon^{2}},
\end{align*}
which goes to $0$ as $N\rightarrow\infty$. This completes the proof.
\end{proof}

Next, we show that the term $\phi\left(\int_0^{t} h(t-s)d \overline{Z}^{N}_s\right)-\phi\left(\int_0^{t} h(t-s)d m_s\right)$ in equation~\eqref{Martingale-eq-0}
can be approximated by $\phi'\left(\int_0^{t} h(t-s)d m_s\right)  \int_0^{t} h(t-s)d \left(\overline{Z}^{N}_s- m_s\right)$ in the sense
in the following lemma which will help us identify the limit of $\left\{\widehat{L}^N,N\geq 1\right\}$ in the proof of Theorem~\ref{CLT-thm}.

\begin{lem}\label{CLT-thm-lem-3}
Suppose (A.1), (A.2)  and   (A.3) hold.  Then for any $\epsilon>0$,
$$
\begin{aligned}
\lim_{N\to\infty}&\mathbb P\bigg(\sup_{t\in[0,T]}\sqrt{N}\bigg|  \phi\left(\int_0^{t} h(t-s)d \overline{Z}^{N}_s\right)-\phi\left(\int_0^{t} h(t-s)d m_s\right)
\\
&\qquad\qquad-\phi'\left(\int_0^{t} h(t-s)d m_s\right)  \int_0^{t} h(t-s)d \left(\overline{Z}^{N}_s- m_s\right)\bigg|\geq \epsilon\bigg)=0.
\end{aligned}
$$
\end{lem}

\begin{proof}
 By the Taylor expansion, we have
$$
\begin{aligned}
& \phi\left(\int_0^{t} h(t-s)d \overline{Z}^{N}_s\right)\\
=&\phi\left(\int_0^{t} h(t-s)dm_s\right)+\int_0^1    \phi'\left(\int_0^{t} h(t-s)d m_s+u \int_0^{t} h(t-s)d \left(\overline{Z}^{N}_s- m_s\right)\right)du\\
&\qquad\qquad\qquad\qquad\qquad\qquad\cdot \int_0^{t} h(t-s)d \left(\overline{Z}^{N}_s- m_s\right).
 \end{aligned}
$$
Set  
\begin{equation}\label{V:eqn}
V_{\phi'}(\delta):=\sup_{0\leq y\leq\delta}\sup_{x\in [0,\infty)} |\phi'(x+y)-\phi'(x)|.
\end{equation}
Since $\phi'$ is uniformly continuous,  $V_{\phi'}(\delta)\to 0$ as $\delta\to 0$.
Note that by integration by parts
\begin{align*}
\left| \int_0^{t} h(t-s)d \left(\overline{Z}^{N}_s- m_s\right)\right|
&=\left| h(0)\left(\overline{Z}^{N}_t- m_t\right)+\int_0^{t} h'(t-s)\left(\overline{Z}^{N}_s- m_s\right)ds\right|
\\
&\leq  \left(h(0)+\| h'\|_{L^{1}[0,T]}\right)\sup_{0\leq s\leq t}\left|\overline{Z}_{s}^{N}-m_{s}\right|.
\end{align*}
Therefore, for any $\delta>0$ and $\epsilon>0$,
\begin{align*}
&\mathbb P\bigg(\sup_{t\in[0,T]}\sqrt{N}\bigg|  \phi\left(\int_0^{t} h(t-s)d \overline{Z}^{N}_s\right)-\phi\left(\int_0^{t} h(t-s)d m_s\right)\\
&\qquad\qquad-\phi'\left(\int_0^{t} h(t-s)d m_s\right)  \int_0^{t-} h(t-s)d \left(\overline{Z}^{N}_s- m_s\right)\bigg|\geq \epsilon\bigg)
\\
&\leq\mathbb{P}\left(\left(h(0)+\| h'\|_{L^{1}[0,T]}\right)\sup_{0\leq s\leq T}\left|\overline{Z}_{s}^{N}-m_{s}\right|
\geq \delta\right)
\\
&\qquad+
\mathbb{P}\left(V_{\phi'}(\delta)   \left(h(0)+\| h'\|_{L^{1}[0,T]}\right)\sqrt{N}\sup_{0\leq s\leq T}\left|\overline{Z}_{s}^{N}-m_{s}\right|
\geq \epsilon\right)
\\
&\leq \left(\frac{1}{\delta^2N}+\frac{(V_{\phi'}(\delta))^2 }{\epsilon^2}\right)
\sup_{N\in\mathbb{N}}N\mathbb{E}\left(\sup_{0\leq t\leq T}
\left|\overline{Z}_{t}^{N}-m_{t}\right|^{2}\right) \left(h(0)+\| h'\|_{L^{1}[0,T]}\right)^{2}.
\end{align*}
By letting first $N\to\infty$ and applying \eqref{N:bar:Z:finite}, and then $\delta\to 0$, we complete  the proof of  Lemma~\ref{CLT-thm-lem-3}.
\end{proof}

Next, let us show that $\left\{\widehat{L}_t^N,N\geq 1\right\}$ is tight in $D([0,T], \mathbb H^{-1}(\mathbb N))$.

\begin{lem}\label{CLT-thm-lem-4}
Suppose (A.1), (A.2)  and   (A.3) hold.  Then $\left\{\widehat{L}_t^N,N\geq 1\right\}$ is tight in  $D([0,T], \mathbb H^{-1}(\mathbb N))$.
 \end{lem}

\begin{proof}
To show the tightness of the sequence  $\left\{\widehat{L}^N, N\geq 1\right\}$ in $D([0,T], \mathbb H^{-1}(\mathbb N))$, it is sufficient to prove that for any $\varphi\in C^{lip}(\mathbb{N})$, \begin{equation}\label{CLT-thm-lem-4-eq-1}
\limsup_{L\rightarrow \infty}\limsup_{N\rightarrow\infty}\mathbb{P}\left(\sup_{t\in[0,T]}\left|\left\<\widehat{L}_t^N,\varphi\right\>\right|\geq L\right)=0,
\end{equation}
and for any $\varphi\in C^{lip}(\mathbb{N})$ and $\varepsilon>0$, 
\begin{equation}\label{CLT-thm-lem-4-eq-2}
\limsup_{\delta\rightarrow 0}\limsup_{N\rightarrow\infty}\mathbb{P}\left(\sup_{s,t\in[0,T],0\leq|t-s|<\delta}\left|\left\<\widehat{L}_t^N-\widehat{L}^N_s,\varphi\right\>\right|\geq \varepsilon\right)=0.
\end{equation}
We only prove \eqref{CLT-thm-lem-4-eq-2} since \eqref{CLT-thm-lem-4-eq-1} is similar. 
  By \eqref{Martingale-eq-0}, 
\begin{equation}\label{CLT-thm-lem-4-eq-5}
\begin{aligned}
\left\<\widehat{L}^N_t,\varphi\right\>
=& M_t^{\varphi,N}+  \int_0^{t}\left\<\widehat{L}_s^N, \nabla \varphi\right\>  \phi\left(\int_0^{s} h(s-u)d m_u\right) ds\\
&+\int_0^{t}\left\<{L}_s^N, \nabla \varphi\right\>   \sqrt{N}\left( \phi\left(\int_0^{s} h(s-u)d \overline{Z}^{N}_u\right)-\phi\left(\int_0^{s} h(s-u)d m_u\right)\right) ds.
\end{aligned}
\end{equation}
It follows  from the proof of Lemma~\ref{CLT-thm-lem-1} that
\begin{equation}\label{CLT-thm-lem-4-eq-6}
\limsup_{\delta\rightarrow 0}\limsup_{N\rightarrow\infty}\mathbb{P}\left(\sup_{s,t\in[0,T],0\leq|t-s|<\delta}\left|M_t^{\varphi,N}-M_s^{\varphi,N}\right|\geq \varepsilon\right)=0.
\end{equation}
Note that for any $0\leq s<t\leq T$,
$$
\begin{aligned}
& \int_s^{t}\left|\left\<{L}_v^N, \nabla \varphi\right\> \right|  \sqrt{N}\left| \phi\left(\int_0^{v} h(v-u)d \overline{Z}^{N}_u\right)-\phi\left(\int_0^{v} h(v-u)d m_u\right)\right| dv\\
\leq & |t-s|\|\varphi\|_{lip}   \sqrt{N}\sup_{v\in[0,T]}\left| \phi\left(\int_0^{v} h(v-u)d \overline{Z}^{N}_u\right)-\phi\left(\int_0^{v} h(v-u)d m_u\right)\right|.
\end{aligned}
$$
By Lemma~\ref{CLT-thm-lem-2}, 
\begin{equation}\label{CLT-thm-lem-4-eq-7}
\begin{aligned}
\limsup_{\delta\rightarrow 0}\limsup_{N\rightarrow\infty}\mathbb{P}&\bigg(\sup_{s,t\in[0,T],0\leq|t-s|<\delta}\bigg|\int_s^{t}\left\<{L}_v^N, \nabla \varphi\right\>  \sqrt{N}\bigg( \phi\left(\int_0^{v} h(v-u)d \overline{Z}^{N}_u\right)\\
&-\phi\left(\int_0^{v} h(v-u)d m_u\right)\bigg) dv\bigg|\geq \varepsilon\bigg)=0.
\end{aligned}
\end{equation}

Since
$$
\begin{aligned}
&\sup_{s,t\in[0,T],0\leq|t-s|<\delta}\left|\int_s^{t}\left\<\widehat{L}_u^N, \nabla \varphi\right\>  \phi\left(\int_0^{u} h(u-v)d m_v\right) du\right|\\
\leq &\delta \left( \phi(0)+\alpha\|h\|_0^Tm_T  \right) \sup_{t\in [0,T]}\left|\left\<\widehat{L}_t^N, \nabla \varphi\right\> \right|,
\end{aligned}
$$
by \eqref{Basic-lem-1-eq-3} and the Chebyshev's inequality, we have that 
 \begin{equation}\label{CLT-thm-lem-4-eq-8}
\limsup_{\delta\rightarrow 0}\limsup_{N\rightarrow\infty}
\mathbb{P}\left(\sup_{s,t\in[0,T],0\leq|t-s|<\delta}\left|\int_s^{t}\left\<\widehat{L}_u^N, \nabla \varphi\right\>  \phi\left(\int_0^{u} h(u-v)d m_v\right) du\right|\geq \varepsilon\right)=0.
\end{equation}
  
Finally, by combining with \eqref{CLT-thm-lem-4-eq-6}, \eqref{CLT-thm-lem-4-eq-7} and \eqref{CLT-thm-lem-4-eq-8}, we obtain \eqref{CLT-thm-lem-4-eq-2}.
This completes the proof.
\end{proof}

Now, we are ready to prove Theorem~\ref{CLT-thm}.

\begin{proof}[Proof of Theorem~\ref{CLT-thm}]  
Let  $X$ be a limit point. Without loss of
generality, we assume that $\widehat{L}^N$ converges weakly  to $X$.  Then  by Lemma~\ref{CLT-thm-lem-1}, it is easy to show that $X$
satisfies \eqref{CLT-thm-eq}. To prove the uniqueness, let  $X$ and  $\widetilde{X}$  be two solutions  of \eqref{CLT-thm-eq} and
define $\overline{X}:=X-\widetilde{X}$.
 Then  for any $t\in [0,T]$,
\begin{equation}\label{unique-CLT-thm-eq}
\begin{aligned}
\left\<\overline{X}_t,\varphi(t)\right\>=  &\int_0^t \left\<\overline{X}_s,\partial_s\varphi(s)\right\>ds+
 \int_0^{t}\left\<\overline{X}_s, \nabla \varphi(s)\right\>  \left(  \phi\left(\int_0^{s} h(s-u)d m_u\right)\right) ds\\
&+\int_0^{t}\left\<\mathcal{L}_s, \nabla \varphi(s)\right\> \phi'\left(\int_0^{s} h(s-u)d m_u\right) \int_0^{s} h(s-u)d\left\<\overline{X}_u, \ell\right\>ds.
\end{aligned}
\end{equation}
First, we take $\varphi(t,x)=\ell(x)$ in \eqref{unique-CLT-thm-eq}.  Then 
\begin{align*}
\left\<\overline{X}_t,\ell\right\>= &\int_0^{t}  \phi'\left(\int_0^{s} h(s-u)d m_u\right) \int_0^{s} h(s-u)d\left\<\overline{X}_u, \ell\right\>ds.
\end{align*}
Set  $G(t):=\sup_{s\in[0,t]}\left|\left\<\overline{X}_s,\ell\right\>\right|$.   Then  by the integration by parts
$$
\left|\int_0^{s} h(s-u)d\left\<\overline{X}_u, \ell\right\>\right|= \left|h(0) \<\overline{X}_s,\ell\>+ \int_0^s \<\overline{X}_u,\ell\> dh(s-u) \right|\leq \left(h(0)+\| h'\|_{L^{1}[0,T]}\right)G(s).
$$
Therefore,
$$
G(t)\leq \left(h(0)+\| h'\|_{L^{1}[0,T]}\right)\int_0^{t} \left| \phi'\left(\int_0^{s} h(s-u)d m_u\right)\right| G(s) ds,~~0\leq t\leq T.
$$
and by the Gronwall inequality,    $G(T)=0$.  Then we have that  for any $t\in [0,T]$,  $\varphi\in C^{lip}(\mathbb N)$,
\begin{equation}\label{unique-CLT-thm-eq-1}
\begin{aligned}
\left\<\overline{X}_t,\varphi\right\>=  &
 \int_0^{t}\left\<\overline{X}_s, \nabla \varphi\right\>   \phi\left(\int_0^{s} h(s-u)d m_u\right)ds.
\end{aligned}
\end{equation}
For each $n\geq0$, take $\varphi(x)=\varphi_n(x):=I_{[n+1,\infty)\cap\mathbb N}(x)$ in \eqref{unique-CLT-thm-eq-1}. Then 
we have $\left\<\overline{X}_t,\varphi\right\>=\sum_{x\geq n+1}\overline{X}_{t}(x)$ and $\left\<\overline{X}_s, \nabla \varphi\right\>=\sum_{x+1\geq n+1}\overline{X}_{s}(x)-\sum_{x\geq n+1}\overline{X}_{s}(x)=\overline{X}_{s}(n)$ such that
\begin{equation}\label{unique-CLT-thm-eq-2}
\begin{aligned}
\sum_{x\geq n+1}\overline{X}_t(x)=  & 
 \int_0^{t} \overline{X}_s(n) \phi\left(\int_0^{s} h(s-u)d m_u\right)ds.
\end{aligned}
\end{equation}
Noting  $\sum_{x\geq 0}\overline{X}_t(x)=0$, by letting $n=0$ in \eqref{unique-CLT-thm-eq-2}, we have that
$$
- \overline{X}_t(0) =   
 \int_0^{t} \overline{X}_s(0)    \phi\left(\int_0^{s} h(s-u)d m_u\right) ds.
$$
Thus, by the Gronwall inequality,   we conclude  $\sup_{t\in [0,T]}|\overline{X}_t(0)|=0$.
By letting $n=1$ in \eqref{unique-CLT-thm-eq-2}, we get
$$
- \overline{X}_t(1)- \overline{X}_t(0)
=- \overline{X}_t(1)   
= \int_0^{t} \overline{X}_s(1)   \phi\left(\int_0^{s} h(s-u)d m_u\right) ds.
$$
By the Gronwall inequality,   we conclude  $\sup_{t\in [0,T]}|\overline{X}_t(1)|=0$.
Recursively,  we can obtain 
$$
-\overline{X}_t(n)=   
 \int_0^{t} \overline{X}_s(n) \phi\left(\int_0^{s} h(s-u)d m_u\right) ds,~~n\geq 0.
$$
Thus, by the Gronwall inequality,   we conclude  $\sup_{t\in [0,T]}|\overline{X}_t(n )|=0$ for all $n\geq 0$. This completes the proof of Theorem~\ref{CLT-thm}.
\end{proof}

\section{Moderate Deviations for the Mean-Field Limit}\label{MDPMeanFieldsSection}

In this section, we study moderate deviations for the mean-field limit. We first show the exponential tightness in Section~\ref{exp:tightness:sec}, 
and then prove the upper bound  and the lower bound in Sections~\ref{upper:bound:MDP} and \ref{lower:bound:MDP} respectively. The proof of the lower bound relies
on the study of perturbed Hawkes processes which is provided in Section~\ref{sec:perturbed}.

\subsection{Exponential tightness}\label{exp:tightness:sec}

In this section, we are going to show the exponential  tightness of  $\left\{\widetilde{L}_t^N,N\geq 1\right\}$.
First of all, by \eqref{central-mu-clt-def-1}, \eqref{central-mu-mdp-def-1} and \eqref{Martingale-eq-0}, we have  that for any $\varphi\in C^{1,lip}([0,T]\times \mathbb N)$,
\begin{equation}\label{Basic-formula-eq-0}
\begin{aligned}
\left\<\widetilde{L}^N_t,\varphi(t)\right\>=& \int_0^t \left\<\widetilde{L}_s^N,\partial_s\varphi(s)\right\>ds\\
&+\int_0^{t}\left\<{L}_s^N, \nabla \varphi(s)\right\>   \frac{\sqrt{N}}{a(N)}\left( \phi\left(\int_0^{s} h(s-u)d \overline{Z}^{N}_u\right)-\phi\left(\int_0^{s} h(s-u)d m_u\right)\right) ds\\
&+ \int_0^{t}\left\<\widetilde{L}_s^N, \nabla \varphi(s)\right\>  \phi\left(\int_0^{s} h(s-u)d m_u\right) ds+\frac{1}{a(N)}M_t^{\varphi, N},
\end{aligned}
\end{equation}
where $M_T^{\varphi, N}$ is a martingale defined by \eqref{Martingale-eq-1} with the predictable quadratic variation defined by \eqref{Martingale-eq-2}.  
We first show that  $\left\{\frac{1}{a(N)}M_t^{\varphi, N}, N\geq 1\right\}$ is exponentially tight.

\begin{lem}\label{MDP-lem-1}
Suppose (A.1), (A.2)  and   (A.3) hold.  Then $\left\{\frac{1}{a(N)}M_t^{\varphi, N}, N\geq 1\right\}$ is exponentially tight in  $D([0,T],\mathbb{R})$ for each $\varphi\in C^{1,lip}([0,T]\times \mathbb N)$.
\end{lem}

\begin{proof}
To show the exponential tightness of the sequence  $\left\{\frac{1}{a(N)}M_t^{\varphi, N}, t\in[0,T]\right\}$ on $D([0,T],\mathbb{R})$ for each $\varphi\in C^{1,lip}([0,T]\times \mathbb N)$, it is sufficient to prove that 
\begin{equation}\label{MDP-lem-1-eq-1}
\limsup_{L\rightarrow \infty}\limsup_{N\rightarrow\infty}\frac{1}{a^2(N)}\log\mathbb{P}\left(\sup_{t\in[0,T]}\left|M_t^{\varphi, N}\right|\geq L a(N)\right)=-\infty,
\end{equation}
and for $\varepsilon>0$, 
\begin{equation}\label{MDP-lem-1-eq-2}
\limsup_{\delta\rightarrow 0}\limsup_{N\rightarrow\infty}\frac{1}{a^2(N)}\log\mathbb{P}\left(\sup_{s,t\in[0,T],0\leq|t-s|<\delta}\left|M_t^{\varphi, N}-M_s^{\varphi, N}\right|\geq \varepsilon a(N)\right)=-\infty.
\end{equation}
We only prove \eqref{MDP-lem-1-eq-2} since  the proof  of \eqref{MDP-lem-1-eq-1} is similar. Without loss of generality, we can assume $T/\delta\in\mathbb{N}$ such that
\begin{equation*}
\begin{aligned}
&\mathbb{P}\left(\sup_{s,t\in[0,T],0\leq|t-s|<\delta}\left|M_t^{\varphi, N}-M_s^{\varphi, N}\right|\geq \varepsilon a(N)\right)\\
&\leq \frac{T}{\delta}\sup_{ 0\leq s\leq T-\delta}\mathbb{P}\left(\sup_{ 0\leq t\leq\delta}\left|M_{t+s}^{\varphi,N}-M^{\varphi, N}_{s}\right|\geq \varepsilon a(N)/2\right).
\end{aligned}
\end{equation*}
Therefore, in order to prove \eqref{MDP-lem-1-eq-2}, it is sufficient to prove that for any $\varepsilon>0$,
\begin{equation}\label{MDP-lem-1-eq-3}
\limsup_{\delta\rightarrow0}\limsup_{N\rightarrow\infty}\frac{1}{a^2(N)}\log\sup_{ 0\leq s\leq T-\delta}\mathbb{P}\left(\sup_{ 0\leq t\leq\delta}\left| M_{t+s}^{\varphi,N}-M^{\varphi,N}_{s}\right|\geq \frac{a(N)\varepsilon}{2}\right)=-\infty.
\end{equation}
Recall that $M_t^{\varphi,N}$ is a martingale defined by \eqref{Martingale-eq-1} with the predictable quadratic variation defined by \eqref{Martingale-eq-2}.     Then  
  there exists a constant $C$ such that  $\left|\Delta M^{\varphi,N}\right|\leq \frac{C}{\sqrt{N}}:=c_N$. 
For each $0\leq s\leq T-\delta$, denote by
 $$
 \widetilde{M}_{t}^{\varphi, N,s}:=M^{\varphi, N}_{t+s}-M^{\varphi, N}_{s},\qquad t\in [0,\delta].
 $$
Then  for any  $t\in[0,\delta]$ and $s\in[0,T-\delta]$,
$$
\begin{aligned}
\left\<\widetilde{M}^{\varphi, N,s}\right\>_t =&\frac{1}{N}\sum_{i=1}^N\int_s^{t+s}  \left(\nabla  \varphi\left(v,Z_{v}^{N,i}\right)\right)^2\phi\left(\int_0^{v}h(v-u)d \overline{Z}^{N}_u\right) dv\\
\leq &\delta\|\nabla \varphi\|^2 \left( \phi(0)+\alpha\|h\|_0^T \overline{Z}^{N}_T  \right),
\end{aligned}
$$
and  for any $\theta\geq 0$,
$$
\exp\left\{\frac{\theta}{c_N} \widetilde{M}_{t}^{\varphi, N,s}-g(\theta/c_N)\left\<\widetilde{M}^{\varphi,N,s}\right\>_t\right\},\qquad t\in [0,\delta],
$$
is a positive  supermartingale (cf. Lemma 26.19 in Kallenberg \cite{Kallenberg2002Book}, Proposition 2 in  \cite{ShorackWellnerBook1986} Appendix B).   Following again the proof of \eqref{Basic-lem-0-eq-10},  there exist positive constants $C_1, C_2, \theta_0$ such that for any $\theta\in [0,\theta_0]$,  $N\geq 1$, $\delta \in [0,T]$,
\begin{equation}\label{MDP-lem-1-eq-5}
\begin{aligned}
&\mathbb{E}\left(e^{ \theta  \sqrt{N}\sup_{0\leq t\leq \delta}\left|\widetilde{M}_t^{\varphi,N,s}\right|}\right) \leq  C_1\delta \theta \sqrt{N} e^{C_2 \delta^2\theta^2 N}+1.
\end{aligned}
\end{equation}
 
Therefore, by the Chebyshev's inequality, we have that
\begin{equation*}
\begin{aligned}
&\limsup_{\delta\to 0}\limsup_{N\rightarrow\infty}\frac{1}{a^2(N)}\log\sup_{ 0\leq s\leq T-\delta}\mathbb{P}\left(\sup_{ 0\leq t\leq\delta}\left| M_{t+s}^{\varphi,N}-M^{\varphi,N}_{s}\right|\geq \frac{a(N)\varepsilon}{2}\right)=-\infty,
\end{aligned}
\end{equation*}
and thus  \eqref{MDP-lem-1-eq-3}  holds.   Then, we obtain \eqref{MDP-lem-1-eq-2}. This completes the proof.
\end{proof}

Next, let us provide an estimate that concerns the second term on the right hand side in \eqref{Basic-formula-eq-0}.
 
\begin{lem}\label{MDP-lem-2}
 Suppose (A.1), (A.2)  and   (A.3) hold. For any $\epsilon>0$,
\begin{equation}\label{MDP-lem-2-eq-1}
\begin{aligned}
\limsup_{L\to\infty}\limsup_{N\to\infty}\frac{1}{a^2(N)}\log \mathbb P\bigg(  &\frac{\sqrt{N}}{a(N)}\sup_{0\leq t\leq T} \bigg|
\phi\bigg(\int_0^{t}
h(t-s)d \overline{Z}_s^{N}\bigg)\\
&\quad   -\phi\left(\int_0^{t} h(t-s)d m_s\right)\bigg|  \geq L\bigg)=-\infty.
\end{aligned}
\end{equation}
\end{lem}

\begin{proof}
Following the proof in Lemma~\ref{CLT-thm-lem-2}, we have
\begin{align*}
&\mathbb P\bigg( \sup_{t\in[0,T]}\bigg| \phi\left(\int_0^{t} h(t-s)d \overline{Z}^{N}_s\right)-\phi\left(\int_0^{t} h(t-s)d m_s\right)\bigg|\geq \epsilon\bigg)
\\
&\leq
\mathbb{P}\left(\|\phi\|_{lip}\left(h(0)+\| h'\|_{L^{1}[0,T]}\right)
\sup_{0\leq t\leq T}\left|\overline{Z}_{t}^{N}-m_{t}\right|
\geq\epsilon\right).
\end{align*}
Then by \eqref{Basic-lem-0-eq-5}, we get \eqref{MDP-lem-2-eq-1}. This completes the proof.
\end{proof}

Now we are finally ready to prove the exponential  tightness of  $\left\{\widetilde{L}_t^N,N\geq 1\right\}$.

\begin{lem}\label{MDP-lem-3}
Suppose (A.1), (A.2)  and   (A.3) hold.  Then $\left\{\widetilde{L}_t^N,N\geq 1\right\}$ is exponentially tight  in  $D([0,T], \mathbb H^{-1}(\mathbb N))$.
\end{lem}

\begin{proof}
To show the exponential tightness of the sequence  $\left\{\widetilde{L}_t^N,N\geq 1\right\}$ on $D([0,T], \mathbb H^{-1}(\mathbb N))$, it is sufficient to prove that for any $\varphi\in C^{lip}(\mathbb{N})$,
\begin{equation}\label{MDP-lem-3-eq-1}
\limsup_{L\rightarrow \infty}\limsup_{N\rightarrow\infty}\frac{1}{a^2(N)}\log\mathbb{P}\left(\sup_{t\in[0,T]}\left|\left\<\widetilde{L}_t^N,\varphi\right\>\right|\geq L\right)=-\infty,
\end{equation}
and for any $\varphi\in C^{lip}(\mathbb{N})$ and $\varepsilon>0$, 
\begin{equation}\label{MDP-lem-3-eq-2}
\limsup_{\delta\rightarrow 0}\limsup_{N\rightarrow\infty}\frac{1}{a^2(N)}\log\mathbb{P}\left(\sup_{s,t\in[0,T],0\leq|t-s|<\delta}\left|\left\<\widetilde{L}_t^N-\widetilde{L}^N_s,\varphi\right\>\right|\geq \varepsilon\right)=-\infty.
\end{equation}

We only show \eqref{MDP-lem-3-eq-2} and the proof for \eqref{MDP-lem-3-eq-1} is similar. 
By  \eqref{Basic-formula-eq-0},  Lemma~\ref{MDP-lem-1} and Lemma~\ref{MDP-lem-2}, it suffices to show  that for any  $\varphi\in C^{lip}(\mathbb N)$,
\begin{equation}\label{MDP-lem-3-eq-5}
\begin{aligned}
\limsup_{M\rightarrow \infty}\limsup_{N\rightarrow\infty}\frac{1}{a^2(N)}\log  \mathbb{P}\left(\sup_{ 0\leq t\leq T} \left|\left\<\widetilde{L}_t^N,\nabla\varphi\right\>\right|\geq M\right)=-\infty.
\end{aligned}
\end{equation}
In fact, \eqref{MDP-lem-3-eq-5}  is just \eqref{Basic-lem-1-eq-2}. 
 Thus,   $\left\{\widetilde{L}_t^N,N\geq 1\right\}$ is exponentially tight  in  $D([0,T], \mathbb H^{-1}(\mathbb N))$ and the proof is complete.
\end{proof}

\subsection{Upper bounds}\label{upper:bound:MDP}
In this subsection,  we show the upper bound of the moderate deviation principle.

Define the exponential local martingale associated with $\<L_t^N,\varphi(t)\>$:
\begin{equation}\label{exp-martingale-eq-1}
\begin{aligned}
{\mathcal E}_t^{N,\varphi}:=&\exp\bigg\{N\left\<L^N_t,\varphi(t)\right\>-N\left\<L^N_0,\varphi(0)\right\>-N\int_0^t\left\<L_s^N,\partial_s \varphi(s)\right\>   ds\\
&-N\int_0^t\left\<L_s^N, \nabla \varphi(s)\right\>   \phi\left(\int_0^{s}h(s-u)d \overline{Z}^{N}_u\right)  ds\\
&\quad - N\int_0^t\left\<L_s^N,e^{\nabla \varphi(s)}-1-\nabla \varphi(s)\right\>  \phi\left(\int_0^{s-}
h(s-u)d \overline{Z}^{N}_u\right) ds\bigg\}\\
=&\exp\bigg\{N\left\<L^N_t,\varphi(t)\right\>-N\left\<L^N_0,\varphi(0)\right\>-N\int_0^t\left\<L_s^N,\partial_s \varphi(s)\right\>   ds\\
&\quad\quad -  N\int_0^t\left\<L_s^N,e^{\nabla \varphi(s)}-1\right\>  \phi\left(\int_0^{s}
h(s-u)d \overline{Z}^{N}_u\right) ds\bigg\}.
\end{aligned}
\end{equation}
If (A.1) and (A.2) hold,
when  $ \|\varphi\|_{lip}$ is sufficiently small, then $\mathcal E_t^{N,\varphi},~t\in [0,T]$ is a martingale.
 Then we denote by
 \begin{equation}\label{mdp-exp-martingale-eq-2}
 \widetilde{\mathcal E}_t^{N,\varphi}:={\mathcal E}_t^{N,\frac{a(N)}{\sqrt{N}}(\varphi(t)-\<\mathcal L_t,\varphi(t)\>)}.
 \end{equation}
 That is, 
  \begin{equation}\label{mdp-exp-martingale-eq-3}
\begin{aligned}
&\widetilde{\mathcal E}_t^{N,\varphi}
=\exp\bigg\{a^2(N)\left\< \widetilde{L}^N_t,\varphi(t)\right\>-a^2(N)\int_0^t\left\< \widetilde{L}_{s}^N,\partial_s \varphi(s)\right\>   ds\\
&\quad -a^2(N)\int_0^t\left\< \widetilde{L}_s^N, \nabla \varphi(s)\right\>   \phi\left(\int_0^s
h(s-u)d m_u\right)  ds\\
&\quad - a^2(N)\int_0^t\left\<  L_s^N, \nabla \varphi(s)\right\>   \frac{\sqrt{N}}{a(N)} \left(\phi\left(\int_0^s
h(s-u)d \overline{Z}^{N}_u\right) -\phi\left(\int_0^s
h(s-u)d m_u\right) \right) ds\\
&\quad - a^2(N)\int_0^t\frac{N}{a^2(N)}\bigg\<  {L}_s^N,e^{\frac{a(N)}{\sqrt{N}}\nabla \varphi(s)}-1-\frac{a(N)}{\sqrt{N}}\nabla \varphi(s)\bigg\>  \phi\left(\int_0^{s}h(s-u)d \overline{Z}^{N}_u\right) ds\bigg\}\,.\\
\end{aligned}
\end{equation}

 We first give  some approximations of the terms in  ${\mathcal E}_T^{N,\varphi}$.

\begin{lem}\label{MDP-lem-4}
Suppose (A.1), (A.2)  and   (A.3) hold. For any $\epsilon>0$,
\begin{equation}\label{MDP-lem-4-eq-1}
\begin{aligned}
&\limsup_{N\rightarrow\infty}\frac{1}{a^{2}(N)}
\log\mathbb P\bigg(\bigg|\int_0^T\left\<  L_t^N, \nabla \varphi(t)\right\>  \frac{\sqrt{N}}{a(N)}\bigg( \phi\left(\int_0^{t} h(t-s)d \overline{Z}^{N}_s\right)\\
&\quad\quad \quad\quad \quad\quad \quad\quad \quad\quad -\phi\left(\int_0^{t} h(t-s)d m_s\right)\bigg)dt\\
&\qquad\qquad-\int_0^T\left\< \mathcal L_t, \nabla \varphi(t)\right\>  \phi'\bigg(\int_0^{t} h(t-s)d m_s\bigg)  \int_0^{t} h(t-s) d\left\<\widetilde{L}_s^N,\ell\right\>dt\bigg|\geq \epsilon\bigg)\\
&
=-\infty.
\end{aligned}.
\end{equation}
\end{lem}

\begin{proof} By \eqref{Basic-lem-1-eq-2} and the Chebyshev's inequality,   
\begin{equation}\label{MDP-lem-4-eq-2}
\limsup_{N\rightarrow\infty}\frac{1}{a^{2}(N)}
\log\mathbb P\left(\sup_{t\in [0,T]}\left|\left\<  L_t^N, \nabla \varphi(t)\right\>-\left\< \mathcal L_t, \nabla \varphi(t)\right\>  \right|\geq \epsilon\right)
=-\infty.
\end{equation}
Following the proof of Lemma~\ref{CLT-thm-lem-3},  for any $\delta>0$,
$$
\begin{aligned}
&\mathbb P\bigg(\sup_{t\in[0,T]}\frac{\sqrt{N}}{a(N)}\bigg|  \phi\left(\int_0^{t} h(t-s)d \overline{Z}^{N}_s\right)-\phi\left(\int_0^{t} h(t-s)d m_s\right)\\
&\qquad\qquad-\phi'\left(\int_0^{t} h(t-s)d m_s\right)  \int_0^{t} h(t-s)d \left(\overline{Z}^{N}_s- m_s\right)\bigg|\geq \epsilon\bigg)
\\
\leq &\mathbb{P}\left(\left(h(0)+\| h'\|_{L^{1}[0,T]}\right)\sup_{0\leq s\leq T}\left|\overline{Z}_{s}^{N}-m_{s}\right|
\geq \delta\right)
\\
&+
\mathbb{P}\left(V_{\phi'}(\delta)   \left(h(0)+\| h'\|_{L^{1}[0,T]}\right)\frac{\sqrt{N}}{a(N)}\sup_{0\leq s\leq T}\left|\overline{Z}_{s}^{N}-m_{s}\right|
\geq \epsilon\right),
\end{aligned}
$$
where $V_{\phi'}(\delta)$ is defined in \eqref{V:eqn}.
By \eqref{Basic-lem-0-eq-3},
$$
\begin{aligned}
&\limsup_{N\to\infty}\frac{1}{a^2(N)}\log\mathbb{P}\left(\left(h(0)+\| h'\|_{L^{1}[0,T]}\right)\sup_{0\leq s\leq T}\left|\overline{Z}_{s}^{N}-m_{s}\right|
\geq \delta\right)=-\infty,
\end{aligned}
$$
 and noting $V_{\phi'}(\delta)\to 0$ as $\delta\to 0$, we also have
$$
\begin{aligned}
&\limsup_{\delta\to 0}\limsup_{N\to\infty}\frac{1}{a^2(N)}\mathbb{P}\left(V_{\phi'}(\delta)   \left(h(0)+\| h'\|_{L^{1}[0,T]}\right)\frac{\sqrt{N}}{a(N)}\sup_{0\leq s\leq T}\left|\overline{Z}_{s}^{N}-m_{s}\right|
\geq \epsilon\right)=-\infty.
\end{aligned}
$$
Therefore,
\begin{equation}\label{MDP-lem-4-eq-3}
\begin{aligned}
\lim_{N\to\infty}\frac{1}{a^2(N)}\log \mathbb P\bigg(& \sup_{0\leq t\leq T}\bigg|\frac{\sqrt{N}}{a(N)}\left(\phi\left(\int_0^{t}
h(t-s)d \overline{Z}_s^N\right)
-\phi\left(\int_0^{t}
h(t-s)d m_s\right)\right)\\
&-  \phi'\left(\int_0^{t} h(t-s)d m_s\right)  \int_0^{t}
h(t-s)d \left\<\widetilde{L}_s^N, \ell\right\>\bigg|\geq \epsilon\bigg)=-\infty.
\end{aligned}
\end{equation}
 Now, by \eqref{MDP-lem-4-eq-2} and \eqref{MDP-lem-4-eq-3}, we obtain \eqref{MDP-lem-4-eq-1}. This completes the proof.
\end{proof}

\begin{lem}\label{MDP-lem-5}
Suppose (A.1), (A.2)  and   (A.3) hold. For any $\epsilon>0$,
$$
\begin{aligned}
\lim_{N\to\infty}\frac{1}{a^2(N)}\log \mathbb P\bigg(  &\bigg|
\frac{N}{a^2(N)} \int_0^T\bigg\<  {L}_t^N,e^{\frac{a(N)}{\sqrt{N}}\nabla \varphi(t)}-1-\frac{a(N)}{\sqrt{N}}\nabla \varphi(t)\bigg\>  \phi\left(\int_0^{t}
h(t-s)d \overline{Z}^{N}_s\right)dt\\
&\quad\quad  - \frac{1}{2}\int_0^T \<  \mathcal L_t, (\nabla \varphi(t))^2\>  \phi\left(\int_0^{t}
h(t-s)d m_s\right) dt \bigg|\geq \epsilon\bigg)=-\infty.
\end{aligned}
$$
\end{lem}

\begin{proof}
Note that $|e^\theta-1-\theta-\frac{\theta^2}{2}| \leq |\theta|^3$ for any $|\theta|\leq 1/2$, and  $\frac{a(N)}{\sqrt{N}}\|\nabla\varphi\|\to 0$ as $N\to\infty$.
There exists $N_0\geq 1$ such that for any $N\geq N_0$,
\begin{equation*}
\begin{aligned}
&\bigg|
\frac{N}{a^2(N)} \int_0^T\bigg\<  {L}_t^N,e^{\frac{a(N)}{\sqrt{N}}\nabla \varphi(t)}-1-\frac{a(N)}{\sqrt{N}}\nabla \varphi(t)\bigg\>  \phi\left(\int_0^{t}
h(t-s)d \overline{Z}^{N}_s\right)\\
&\qquad  - \frac{1}{2}\int_0^T\<  {L}_t^N, (\nabla \varphi(t))^2\>  \phi\left(\int_0^{t}
h(t-s)d m_s\right) dt \bigg|\\
&\leq T\left(\|\nabla \varphi\|^2+\frac{a(N)}{\sqrt{N}}\|\nabla\varphi\|^3\right)\sup_{t\in[0,T]}\left|\phi\left(\int_0^{t}
h(t-s)d m_s\right)-\phi\left(\int_0^{t}h(t-s)d \overline{Z}^{N}_s\right)\right|\\
&\qquad\qquad  +\frac{a(N)}{\sqrt{N}}\ \|\nabla\varphi\|^3  \sup_{t\in[0,T]}\left|\phi\left(\int_0^{t}h(t-s)dm_s\right)\right|. 
\end{aligned}
\end{equation*}
It is obvious that
$$
\frac{a(N)}{\sqrt{N}}\ \|\nabla\varphi\|^3  \sup_{t\in[0,T]}\left|\phi\left(\int_0^{t}h(t-s)dm_s\right)\right|\to 0 \mbox{ as }N\to\infty.
$$
By \eqref{Basic-lem-1-eq-2} and the Chebyshev's inequality,   
\begin{equation}\label{MDP-lem-5-eq-2}
\limsup_{N\rightarrow\infty}\frac{1}{a^{2}(N)}
\log\mathbb P\left(\sup_{t\in [0,T]}\left|\left\<  L_t^N, (\nabla \varphi(t))^2\right\>-\left\< \mathcal L_t, (\nabla \varphi(t))^2\right\>  \right|\geq \epsilon\right)
=-\infty.
\end{equation}

By Lemma~\ref{MDP-lem-2}, we have
\begin{equation}\label{MDP-lem-5-eq-3}
\begin{aligned}
\lim_{N\to\infty}\frac{1}{a^2(N)}\log\mathbb P\bigg( \sup_{t\in[0,T]}\left|\phi\left(\int_0^{t}
h(t-s)d m_s\right)-\phi\left(\int_0^{t-}
h(t-s)d \overline{Z}^{N}_s\right)\right|\geq \epsilon\bigg)  =-\infty.
\end{aligned}
\end{equation}
Therefore,  by  \eqref{MDP-lem-5-eq-2} and    \eqref{MDP-lem-5-eq-3},  the conclusion of the lemma holds.
\end{proof}

Now, we prove the upper bound of the moderate deviation principle.

\begin{thm}\label{Up-bound-thm}
Suppose (A.1), (A.2) and (A.3) hold. Then for any closed subset $C\subset D([0,T], \mathbb H^{-1}(\mathbb N))$,
\begin{equation}\label{Up-bound-thm-eq-1}
\limsup_{N\rightarrow\infty}\frac{1}{a^2(N)}\log\mathbb{P}\left(\widetilde{L}^N\in C\right)\leq -\inf_{\mu\in C}I(\mu).
\end{equation}
\end{thm}

\begin{proof} For each $L>0$, by Lemma~\ref{MDP-lem-3}, we can  choose a compact set $K_L \subset D([0,T], \mathbb H^{-1}(\mathbb N))$ such that
\begin{equation}\label{Up-bound-thm-eq-2}
\limsup_{N\rightarrow\infty}\frac{1}{a^2(N)}\log\mathbb{P}\left(\widetilde{L}^N\in K_L^c\right)\leq -L.
\end{equation}
Set $C_L:=C\cap K_L$. Then $C_L$ is compact. Define
$$
I(C_L)=\inf_{\mu\in C_L}I(\mu)=\inf_{\mu\in C_L}\sup\left\{J_\mu(\varphi):\varphi\in C^{1,lip}([0,T]\times\mathbb{N})\right\}. 
$$
Then for any $\mu\in C_L$ and $\epsilon>0$, there exists    $\varphi_{\mu}\in C^{1,lip}([0,T]\times\mathbb{N})$ such that
$$
J_\mu(\varphi_{\mu})\geq I(C_L)-\frac{\epsilon}{2},
$$
Next, we choose an open  neighborhood $G_{\mu}$ of $\mu$ such that for any $\nu\in G_{\mu}$,
$$
J_\nu(\varphi_{\mu})\geq I(C_L)- \epsilon.
$$
Let $G_{\mu_1},\ldots,G_{\mu_k}$  be  a finite open covering of $C_L$.   
For any $\varphi\in C^{1,lip}([0,T]\times\mathbb{N})$, set
\begin{equation}\label{A1-set}
\begin{aligned}
A_1(\varphi,\epsilon):=& \bigg\{\bigg|\int_0^T\left\<  L_t^N, \nabla \varphi(t)\right\>  \frac{\sqrt{N}}{a(N)}\bigg( \phi\left(\int_0^{t} h(t-s)d \overline{Z}^{N}_s\right) -\phi\left(\int_0^{t} h(t-s)d m_s\right)\bigg)dt\\
&\qquad-\int_0^T\< \mathcal L_t, \nabla \varphi(t)\>  \phi'\bigg(\int_0^{t} h(t-s)d m_s\bigg)  \int_0^{t} h(t-s) d\left\<\widetilde{L}_s^N,\ell\right\>dt\bigg|\geq \epsilon \bigg\},
\end{aligned}
\end{equation}
and
\begin{equation}\label{A2-set}
\begin{aligned}
A_2(\varphi,\epsilon):= &\bigg\{\bigg|\frac{N}{a^2(N)} \int_0^T\bigg\<  {L}_t^N,e^{\frac{a(N)}{\sqrt{N}}\nabla \varphi(t)}-1-\frac{a(N)}{\sqrt{N}}\nabla \varphi(t)\bigg\>  \phi\left(\int_0^{t}
h(t-s)d \overline{Z}^{N}_s\right)  \\
&\qquad\qquad   - \frac{1}{2}\left\<  \mathcal L_t, (\nabla \varphi(t))^2\right\>  \phi\left(\int_0^{t}
h(t-s)d m_s\right) dt \bigg|\geq \epsilon \bigg\},
\end{aligned}
\end{equation}
and  define $A(\varphi,\epsilon):= A_1(\varphi,\epsilon)\cup A_2(\varphi,\epsilon)$. Then by Lemma~\ref{MDP-lem-4} and \ref{MDP-lem-5}, we have
\begin{equation*}
\begin{aligned}
\limsup_{N\rightarrow\infty}\frac{1}{a^2(N)}\log\mathbb{P}(A(\varphi, \epsilon))=-\infty.
\end{aligned}
\end{equation*}
Note that on  $\left\{\widetilde{L}^N\in G_\mu\right\}\cap (A(\varphi,\epsilon))^c$,
$$
\begin{aligned}
\left(\widetilde{\mathcal E}_T^{N,\varphi}\right)^{-1}
&\leq\exp\bigg\{-a^2(N)\inf_{\nu\in G_\mu}\bigg(\< \nu_T,\varphi(T)\>-\int_0^T\<\nu_{s},\partial_s \varphi(s)\>   ds\\
&\quad -\int_0^T\< \nu_s, \nabla \varphi(s)\>   \phi\left(\int_0^s
h(s-u)d m_u\right)  ds\\
&\quad\quad -\int_0^T\<  \nu_t, \nabla \varphi(t)\>    \phi'\left(\int_0^{t} h(t-s)d m_s\right)  \int_0^{t}
h(t-s)d \<\nu_s, \ell\>dt\\
&\quad\quad\quad -\frac{1}{2}\int_0^T \left\<  \mathcal L_t, (\nabla \varphi(t))^2\right\>  \phi\left(\int_0^{t}
h(t-s)d m_s\right) dt\bigg)+2\epsilon a^2(N) \bigg\}.
\end{aligned}
$$
Then for any $\varphi\in C^{1,lip}([0,T]\times\mathbb{N})$, 
\begin{equation*}
\begin{aligned}
& \limsup_{N\rightarrow\infty}\frac{1}{a^2(N)}\log\mathbb{P}\left(\widetilde{L}^N\in G_{\mu_i}\right)\\
&= \limsup_{N\rightarrow\infty}\frac{1}{a^2(N)}\log\mathbb{P}\left(\left\{\widetilde{L}^N\in G_{\mu_i}\right\}\cap A^c(\mu_i,\epsilon)\right)\\
&= \limsup_{N\rightarrow\infty}\frac{1}{a^2(N)}\log\mathbb{E}\left(\left(\widetilde{\mathcal E}_T^{N,\varphi}\right)^{-1}\widetilde{\mathcal E}_T^{N,\varphi}I_{\{\widetilde{L}^N\in G_{\mu_i}\}\cap (A(\varphi,\epsilon))^c}\right)\\
&\leq -  \inf_{\nu\in G_{\mu_i}}J_{\nu}(\varphi)+2\epsilon.
\end{aligned}
\end{equation*}
Thus,
\begin{equation}\label{MDP-upper bound-eq-1}
\begin{aligned}
&\limsup_{N\rightarrow\infty}\frac{1}{a^2(N)}\log\mathbb{P}\left(\widetilde{L}^N\in G_{\mu_i}\right)\leq -\sup_{\varphi\in C^{1,lip}([0,T]\times\mathbb{N})}\inf_{\nu\in G_{\mu_i}}J_\nu(\varphi)+2\epsilon.
\end{aligned}
\end{equation}
By \eqref{MDP-upper bound-eq-1}, we obtain that
\begin{equation*}
\begin{aligned}
\limsup_{N\rightarrow\infty}\frac{1}{a^2(N)}\log\mathbb{P}\left(\widetilde{L}^N\in C_L\right)\leq &-\min_{1\leq i\leq k}\sup_{\varphi\in C^{1,lip}([0,T]\times\mathbb{N})}\inf_{\nu\in G_{\mu_i}}J_\nu(\varphi)+2\epsilon\\
\leq & -\ \min_{1\leq i\leq k} \inf_{\nu\in G_{\mu_i}}J_\nu(\varphi_{\mu_i})+2\epsilon\\
\leq & -I(C_L)+3\epsilon.
\end{aligned}
\end{equation*}
By letting $\epsilon\to 0$, we obtain
$$
\limsup_{N\rightarrow\infty}\frac{1}{a^2(N)}\log\mathbb{P}\left(\widetilde{L}^N\in C_L\right)\leq  -I(C_L).
$$ 
Thus
$$
\begin{aligned}
\limsup_{N\rightarrow\infty}\frac{1}{a^2(N)}\log\mathbb{P}\left(\widetilde{L}^N\in C\right)
&\leq \max\left\{\limsup_{N\rightarrow\infty}\frac{1}{a^2(N)}\log\mathbb{P}\left(\widetilde{L}^N\in C_L\right),  -L\right\}\\
&\leq - \min\left\{ L, \inf_{\mu \in C}I(\mu)\right\}.
\end{aligned}
$$
Finally, by letting $L\to\infty$,  we complete the proof of the upper bound.
\end{proof}

\subsection{Perturbed Hawkes process}\label{sec:perturbed}

In order to prove the lower bound for the moderate deviation principle, we need to introduce and provide law of large numbers for perturbed Hawkes processes.
For each $\psi\in C^{1,lip}([0,T]\times\mathbb{N})$, we introduce the perturbed Hawkes precess $\left(Z_t^{\psi,N,1},\ldots,Z_t^{\psi,N,N}\right)_{t\geq 0}$ defined by the following SDEs:
\begin{equation}\label{Pertur-Hawkes-psi-eq}
Z^{\psi,N,i}_t=\int_0^t \int_0^\infty I_{\left\{z \leq e^{\frac{a(N)}{\sqrt{N}}\nabla\psi(s,Z^{\psi,N,i}_{s-})}\phi\left(\int_0^{s-}h(s-u)d\overline{Z}_u^{\psi, N}\right)\right\}}
\pi^i(ds\,dz),
\qquad
1\leq i\leq N,
\end{equation}
where $\overline{Z}_t^{\psi, N}:=\frac{1}{N}\sum_{i=1}^NZ^{\psi,N,i}_t$.
It is known that there exist a unique solution to \eqref{Pertur-Hawkes-psi-eq} (see Lemma~2.1 in \cite{GaoFZhuMF-LDP2023}).

 Define $L_t^{\psi,N}(x):=\frac{1}{N}\sum_{i=1}^N\delta_{Z^{\psi,N,i}_t}(\{x\})$, and 
\begin{equation*}\label{Pertur-Hawkes-M-MF-def}
 \widetilde{L}_t^{\psi,N}(x) := \frac{1}{\sqrt{N}a(N)}\sum_{i=1}^N
\left( \delta_{Z^{\psi,N,i}_t}(\{x\})-\mathcal L_t(\{x\})\right).
\end{equation*}

Let  $\mathbb B([0,T]\times \mathbb N)$ denote the space of all real measurable functions on  $[0,T]\times \mathbb N$. Set 
$$
\mathbb L^2:=\left\{g\in \mathbb B([0,T]\times \mathbb N);~\|g\|_{\mathbb L^2}^2=\int_0^T\sum_{x\in\mathbb N }|g(t,x)|^2 \mathcal L_t(x)  \phi\left(\int_0^{t}
h(t-s)d m_s\right) dt<\infty \right\}.
$$
For  $f,g\in \mathbb L^2$, $g\sim f$ means $\|f-g\|_{\mathbb L^2}=0$.  
For simplicity of notations, we still denote by  $ \mathbb L^2= \mathbb L^2/\sim $. Then 
 $\mathbb L^2$ is  a Hilbert space.

For any  $g\in\mathbb L^2$, we consider the differential equation:
\begin{equation}\label{Pertur-lim-equation-eq-1}
\left\{
\begin{aligned}
\partial_t \mu_{t}+& \phi\left(\int_0^{t} h(t-s)d m_s\right) \nabla \mu_t 
+ \phi'\left(\int_0^{t}h(t-s)d m_s\right)  \int_0^{t}h(t-s)d \<\mu_s,\ell\>\nabla\mathcal L_t \\
&+g(t)  \phi\left(\int_0^{t}h(t-s)d m_s\right)\nabla\mathcal L_t=0,\\
\mu_0=&0.
\end{aligned}
\right.
\end{equation}
A function $\mu=\{\mu_t(x), t\in[0,T],x\in\mathbb N\}\in D([0,T], \mathbb H^{-1}(\mathbb N))$  is said to be a solution to the  equation  \eqref{Pertur-lim-equation-eq-1} if  for any $ \varphi\in C^{1,lip}([0,T]\times\mathbb{N})$,
\begin{equation}\label{Pertur-lim-equation-eq-2}
\begin{aligned}
&\< \mu_T,\varphi(T)\>- \int_0^{T}\< \mu_{t},\partial_t \varphi(t)\>   dt-\int_0^{T}\< \mu_t, \nabla \varphi(t)\>   \phi\left(\int_0^{t} h(t-s)d m_s\right)  dt\\
&\qquad-\int_0^{T}\< \mathcal L_t, \nabla \varphi(t)\> \phi'\left(\int_0^{t}
h(t-s)d m_s\right)  \int_0^{t}
h(t-s)d \<\mu_s,\ell\>dt\\
&\qquad\qquad-\int_0^T\<g(t) \mathcal L_t, \nabla \varphi(t) \>  \phi\left(\int_0^{t}
h(t-s)d m_s\right) dt=0.
\end{aligned}
\end{equation}

Next, we show that the solution to the  equation  \eqref{Pertur-lim-equation-eq-1} is unique, 
and the proof of the existence of the solution to \eqref{Pertur-lim-equation-eq-1} will be provided later in Theorem~\ref{Pertur-Hawkes-LLN-thm}. 

\begin{lem}\label{Pertur-equation-uniq-lem}
Suppose that (A.1), (A.2) and (A.3) hold.  Then for each $g\in \mathbb L^2$,  the solution to the  equation  \eqref{Pertur-lim-equation-eq-1} is unique. 
\end{lem}

\begin{proof}
Let $\nu$ and $\mu$ be two solutions to \eqref{Pertur-lim-equation-eq-1}.
Ser $\overline{\mu}:=\nu-\mu$.
 Then  for any $t\in [0,T]$,
\begin{equation}\label{Pertur-lim-equation-eq-2}.
\begin{aligned}
\left\<\overline{\mu}_t,\varphi(t)\right\>=  &\int_0^t \left\<\overline{\mu}_s,\partial_s\varphi(s)\right\>ds+
 \int_0^{t}\left\<\overline{\mu}_s, \nabla \varphi(s)\right\>  \left(  \phi\left(\int_0^{s} h(s-u)d m_u\right)\right) ds\\
&+\int_0^{t}\left\<\mathcal{L}_s, \nabla \varphi(s)\right\> \phi'\left(\int_0^{s} h(s-u)d m_u\right) \int_0^{s} h(s-u)d\left\<\overline{\mu}_u, \ell\right\>ds.
\end{aligned}
\end{equation}
Thus, by following the same argument as in  the proof  of  uniqueness  in Theorem \ref{CLT-thm}  yields  $\overline{\mu}=0$. 
This completes the proof.
\end{proof}

The following result is a law of large numbers for the mean-fields  of the perturbed Hawkes precess.
 
\begin{thm}\label{Pertur-Hawkes-LLN-thm}
Suppose that (A.1), (A.2) and (A.3) hold.  Let  $\psi\in C^{1,lip}([0,T]\times\mathbb{N})$ and let  $\left\{Z_t^{\psi,N,i},t\in[0,T]\right\}$ be a solution of the SDEs \eqref{Pertur-Hawkes-psi-eq}. 
Then there exists a unique solution  $\mu^\psi=\left\{\mu_t^{\psi}(x), t\in[0,T],x\in\mathbb N\right\}$ to the  equation  \eqref{Pertur-lim-equation-eq-1},
and  for any an open set $O\ni \mu^{\psi}$,
\begin{align}
\lim_{N\rightarrow\infty}\mathbb{P}\left(\widetilde{L}^{N,\psi}\in O\right)=1.
\end{align}
\end{thm}

\begin{proof}
For any $\varphi\in C^{1,lip}([0,T]\times\mathbb{N})$, by It\^{o}'s formula,
\begin{equation}\label{MDP-New-Ito formula}
\begin{aligned}
&\left\<\widetilde{L}^{N,\psi}_T,\varphi(T)\right\>
\\
=&\int_0^{T}\left\< \widetilde{L}^{N,\psi}_t,\partial_t \varphi(t)\right\>   dt+ \int_0^{T}\left\<\widetilde{L}_t^{N,\psi}, \nabla \varphi(t)\right\>  \phi\left(\int_0^{t} h(t-s)d m_s\right) dt\\
&+ \int_0^{T}\left\<L_t^{N,\psi}, \nabla \varphi(t)\right\>   \frac{\sqrt{N}}{a(N)}\left( \phi\left(\int_0^{t} h(t-s)d \overline{Z}^{\psi,N}_s\right)-\phi\left(\int_0^{t} h(t-s)d m_s\right)\right) dt\\
&+ \frac{\sqrt{N}}{a(N)}\int_0^{T}\left\<L_t^{N,\psi}, \nabla \varphi(t)\left(e^{\frac{a(N)}{\sqrt{N}}\nabla\psi(t)}-1\right)\right\>   \phi\left(\int_0^{t} h(t-s)d \overline{Z}^{\psi,N}_s\right) dt+ \frac{\sqrt{N}}{a(N)}M_T^{N,\psi}\,,\\
\end{aligned}
\end{equation}
where 
\begin{align*}
M_T^{N,\psi}&:=\frac{1}{N}\sum_{i=1}^N\int_0^{T} \int_0^\infty \nabla  \varphi\left(t,Z_{t-}^{\psi,N,i}\right)
\\
&\qquad\qquad\qquad\qquad\cdot I_{ \left\{z \leq e^{\frac{a(N)}{\sqrt{N}}\nabla\psi(t,Z^{\psi,N,i}_{t-})}\phi\left(\int_0^{t-}
h(t-s)d \overline{Z}^{\psi,N}_s\right)\right\}}   ({\pi}^{i}(dt\,dz)-dtdz)
\end{align*}
is a martingale with the predictable quadratic variation
\begin{equation*}
\begin{aligned}
\left\<M^{N,\psi}\right\>_T:=& \int_0^{T} \left\<L_t^{N,\psi},( \nabla \varphi(t))^2 e^{\frac{a(N)}{\sqrt{N}}\nabla\psi(t)}\right\>  \phi\left(\int_0^{t}
h(t-s)d \overline{Z}^{\psi,N}_s\right)  dt\,.
\end{aligned}
\end{equation*}
Note that $L_t^{N,\psi}$ converges to the mean-field limit $\mathcal{L}$. Following the same argument as in the proof of Lemma~\ref{CLT-thm-lem-1} and \ref{CLT-thm-lem-4}, we can show that $\left\{\widetilde{L}^{N,\psi},N\geq 1\right\}$ is tight, and any of its limit point $\mu$ satisfies \eqref{Pertur-lim-equation-eq-1}.  Therefore, there exists a solution to the equation \eqref{Pertur-lim-equation-eq-1}.   By  Lemma~\ref{Pertur-equation-uniq-lem},    the equation \eqref{Pertur-lim-equation-eq-1} has a unique solution $\mu^\psi=\left\{\mu_t^{\psi}(x), t\in[0,T],x\in\mathbb N\right\}$, and $\left\{\widetilde{L}^{N,\psi},N\geq 1\right\}$ converges weakly to $\mu^\psi=\left\{\mu_t^{\psi}(x), t\in[0,T],x\in\mathbb N\right\}$. The proof is complete.
\end{proof}

\subsection{Lower bound}\label{lower:bound:MDP}

In this subsection, we prove the lower bound of the moderate deviation principle.
First, we define the scalar product in $C^{1,lip}([0,T]\times\mathbb{N})$ as
$$
[f,g]:=\int_0^T\< \mathcal L_t, \nabla f(t) \nabla g(t) \>  \phi\left(\int_0^{t}
h(t-s)d m_s\right) dt.
$$
For $f,g\in C^{1,lip}([0,T]\times\mathbb{N})$, $g\sim f$ means $[f-g,f-g]=0$. 
Let $\mathbb{H}^1$ denote the Hilbert
space defined as completion of $C^{1,lip}([0,T]\times\mathbb{N})/\sim$ with respect to the scalar product $[\cdot,\cdot]$.
It is obvious that
$$
\psi\in \mathbb{H}^1  \mbox{ if and only if } \nabla \psi \in \mathbb L^2,
$$
and
$$
\|\psi\|_{\mathbb H^1}=\|\nabla\psi\|_{\mathbb L^2}.
$$

Next, we establish some properties for the rate function ${I}(\mu)$.

\begin{lem}\label{MDP-rate-f-property-lem}
Suppose that (A.1), (A.2) and (A.3) hold. If ${I}(\mu)<\infty$, then there exists a function $\psi\in\mathbb{H}^1$ such that
\begin{equation}\label{MDP-rate-f-property-lem-eq-1-1}
{I}(\mu)=\frac{1}{2}[\psi,\psi]=\frac{1}{2} \|\nabla\psi\|_{\mathbb L^2}.
\end{equation}
Furthermore, let $\psi^{(n)}\in C^{1,lip}([0,T]\times\mathbb{N})$, $n\geq1$ satisfy $\left[\psi^{(n)}-\psi,\psi^{(n)}-\psi\right]\rightarrow0$ as $n\rightarrow\infty$ and let $\mu^{(n)}$  be the solution of \eqref{Pertur-lim-equation-eq-1} associated with $\nabla\psi^{(n)}$.  Then 
\begin{equation}\label{MDP-rate-f-property-lem-eq-1-2}
\mu^{(n)}\rightarrow\mu,\quad\quad {I}\left(\mu^{(n)}\right)\rightarrow{I}(\mu).
\end{equation}
\end{lem}

\begin{proof}
By the definition of $I(\mu)$, 
$$
{I}(\mu)=\sup\left\{{\Upsilon}_\mu(\varphi)-\frac{1}{2}[\varphi,\varphi];\varphi\in C^{1,lip}([0,T]\times\mathbb{N})\right\}, 
$$
we have 
$
|{\Upsilon}_\mu(\varphi)|^2\leq 2{I}(\mu)[\varphi,\varphi]
$  for any  
$\varphi\in C^{1,lip}([0,T]\times\mathbb{N}),
$
and thus ${\Upsilon}_\mu$ can be extended to $\mathbb{H}^1$ and
$$
|{\Upsilon}_\mu(\varphi)|^2\leq 2{I}(\mu)[\varphi,\varphi],~~~\varphi\in \mathbb{H}^1.
$$
By the Riesz representation theorem, there exists $\psi\in\mathbb{H}^1$ such that 
\begin{align}\label{MDP-rate-f-property-lem-eq-2}
{\Upsilon}_\mu(\varphi)=[\psi,\varphi],~~~~\varphi\in C^{1,lip}([0,T]\times\mathbb{N}).
\end{align}
Thus,  \eqref{MDP-rate-f-property-lem-eq-1-1} holds.  Furthermore,   $\mu$  is the solution of \eqref{Pertur-lim-equation-eq-1} associated with $\nabla\psi$.

Now, let $\psi^{(n)}\in C^{1,lip}([0,T]\times\mathbb{N})$, $n\geq1$ satisfy $\left[\psi^{(n)}-\psi,\psi^{(n)}-\psi\right]\rightarrow0$ as $n\rightarrow\infty$ and let $\mu^{(n)}$  be the solution of \eqref{Pertur-lim-equation-eq-1} associated with $\nabla\psi^{(n)}$.  
Set $\overline{\mu}^{(n)}:=\mu^{(n)}-\mu$.
 Then  
 for any $\varphi\in C^{1,lip}([0,T]\times\mathbb{N})$ and $t\in [0,T]$,
\begin{equation}\label{MDP-rate-f-property-lem-eq-3}
\begin{aligned}
\left\<\overline{\mu}_t^{(n)},\varphi(t)\right\>=  &\int_0^t \left\<\overline{\mu}_s^{(n)},\partial_s\varphi(s)\right\>ds+
 \int_0^{t}\left\<\overline{\mu}_s^{(n)}, \nabla \varphi(s)\right\>   \phi\left(\int_0^{s} h(s-u)d m_u\right) ds\\
&+\int_0^{t}\left\<\mathcal{L}_s, \nabla \varphi(s)\right\> \phi'\left(\int_0^{s} h(s-u)d m_u\right) \int_0^{s} h(s-u)d\left\<\overline{\mu}_u^{(n)}, \ell\right\>ds\\
&+\int_0^t\left\< \mathcal L_s, \left(\nabla \psi^{(n)}(s)-\nabla\psi(s)\right) \nabla \varphi(s) \right\>  \phi\left(\int_0^{s}h(s-u)d m_u\right) ds.
\end{aligned}
\end{equation}
By the  Cauchy-Schwarz inequality,
$$
\left|\left\< \mathcal L_s, \left(\nabla \psi^{(n)}(s)-\nabla\psi(s)\right) \nabla \varphi(s) \right\> \right|\leq \left\< \mathcal L_s, | \nabla \varphi(s) |^2\right\>^{1/2} \left\< \mathcal L_s, \left(\nabla \psi^{(n)}(s)-\nabla\psi(s)\right)^2\right\>^{1/2}.
$$
Then, by following the proof of uniqueness in  Theorem~\ref{CLT-thm}, 
one can show that there exist positive constants $c$ and $C$  such that 
\begin{equation}\label{MDP-rate-f-property-lem-eq-4}
G^{(n)}(t) \leq c\int_0^t G^{(n)}(s)ds +C\left[\psi^{(n)}-\psi,\psi^{(n)}-\psi\right]^{1/2},
\end{equation}
where $G^{(n)}(t)= \sup_{s\in[0,t]}\left|\left\<\overline{\mu}_s^{(n)}, \ell\right\>\right|$. By the Gronwall inequality, 
$$
G^{(n)}(t) \leq C\left[\psi^{(n)}-\psi,\psi^{(n)}-\psi\right]^{1/2}e^{ct}.
$$
Noting $\overline{\mu}^{(n)}\in D([0,T]\times \mathbb H^{-1}(\mathbb N))$,   we have
$$
\eta_t^{(n)}:=\sup_{0\leq s\leq t}\sum_{x\in\mathbb N}\left|\overline{\mu}^{(n)}_s(x)\right|<\infty,~t\in[0,T].
$$
Using  \eqref{MDP-rate-f-property-lem-eq-3} and  \eqref{MDP-rate-f-property-lem-eq-4},    there exist positive constants $c_1$ and $C_1$  such that for any   $t\in [0,T]$,
\begin{equation}\label{MDP-rate-f-property-lem-eq-5}
\sup_{ \varphi\in C^{lip}(\mathbb{N}) \text{ with } |\nabla \varphi|\leq 1 }\left|\left\<\overline{\mu}_t^{(n)},\varphi\right\>\right|\leq   c_1\int_0^t \eta_s^{(n)}ds+C_1\left[\psi^{(n)}-\psi,\psi^{(n)}-\psi\right]^{1/2},
\end{equation}
Note that $\{\frac{1}{4}\varphi;~\varphi(x)\in \{-1,1\},~x\in\mathbb N\}\subset \{ \varphi\in C^{lip}(\mathbb{N});~ |\nabla \varphi|\leq 1\}$.
For each fixed $t\in [0,T]$, define $\varphi_n(x):= {\rm sign}\left(\overline{\mu}^{(n)}_t(x)\right)$. Then  by 
$$
\sum_{x\in\mathbb N}\left|\overline{\mu}^{(n)}_t(x)\right|= 4\left\<\overline{\mu}^{(n)}_t,\frac{1}{4}\varphi_n\right\>,
$$  
and \eqref{MDP-rate-f-property-lem-eq-5}, we have
$$
\sum_{x\in\mathbb N}\left|\overline{\mu}^{(n)}_t(x)\right|\leq 4 c_1\int_0^t\eta_s^{(n)}ds+4C_1\left[\psi^{(n)}-\psi,\psi^{(n)}-\psi\right]^{1/2}.
$$
Thus, there exist positive constants $c_2$ and $C_2$  such that 
$$
\eta_t^{(n)}=\sup_{0\leq s\leq t}\sum_{x\in\mathbb N}\left|\overline{\mu}^{(n)}_s(x)\right| \leq c_2\int_0^t\eta_s^{(n)}ds+C_2\left[\psi^{(n)}-\psi,\psi^{(n)}-\psi\right]^{1/2}.
$$
Using the Gronwall inequality again,  we have
$$
\eta_t^{(n)}\leq C_2\left[\psi^{(n)}-\psi,\psi^{(n)}-\psi\right]^{1/2}e^{c_2t}.
$$
Therefore, 
$$
\lim_{n\to\infty} \sup_{t\in[0,T]}\max\left\{G^{(n)}(t),\eta_t^{n}\right\}=0.
$$
Finally,  for any $\varphi\in C^{lip}(\mathbb N)$,   by  \eqref{MDP-rate-f-property-lem-eq-5} and $\nabla (\varphi/\|\varphi\|_{lip})\leq 1$, 
$$
\lim_{n\to\infty }\sup_{s\in [0,T]}\left|\left\<\overline{\mu}_s^{(n)}, \varphi\right\>\right|=0.
$$
This yields that as $n\rightarrow\infty$,
$
\mu^{(n)}\rightarrow\mu, 
$
and
$$
{I}\left(\mu^{(n)}\right)=\frac{1}{2}\left[\psi^{(n)},\psi^{(n)}\right]\to \frac{1}{2}[\psi,\psi]={I}(\mu).
$$
This completes the proof.

\end{proof} 

Now, we are finally ready to prove the lower bound for the moderate deviation principle.

\begin{thm}
Suppose that (A.1), (A.2) and (A.3) hold. Then for any $\mu\in D([0,T], \mathbb H^{-1}(\mathbb N))$ and open set $O\ni\mu$,
\begin{align}
\liminf_{N\rightarrow \infty}\frac{1}{a^2(N)}\log\mathbb{P}\left(\widetilde{L}^N\in O\right)\geq -{I}(\mu).
\end{align}
\end{thm}

\begin{proof}
We can assume ${I}(\mu)<\infty$.   Choose  $\epsilon_0>0$ small enough such that $B(\mu,\epsilon_0)\subset O$, where $B(\mu,\epsilon_0)$ is the open ball centered at $\mu$ of radius $\epsilon_0$.  Then by Lemma~\ref{MDP-rate-f-property-lem}, for any $\epsilon\in (0,\epsilon_0)$, we can choose $\psi^{\epsilon}\in C^{1,lip}([0,T]\times\mathbb{N})$  such that 
$$
\mu^{\epsilon}\in B(\mu,\epsilon),\quad  |{I}(\mu^{\epsilon})-{I}(\mu)|<\epsilon,   
$$
where  $\mu^{\epsilon}\in D([0,T], \mathbb H^{-1}(\mathbb N))$ is the solution of \eqref{Pertur-lim-equation-eq-1} associated with $\psi^{\epsilon}$.  

Let $\mathbb{P}^{0,N}$ be the probability distribution of $\left\{\left(Z^{N,1}_t,\cdots,Z^{N,N}_t\right),t\in [0,T]\right\}$, and we also let $\left\{\left(Z^{N,1}_t,  \cdots, Z^{N,N}_t\right),   t\in [0,T]\right\}$   also denote the coordinate process  in $ D([0,T], \mathbb N^N)$. 
Furthermore, let  $\left\{\left(Z^{\psi^{\epsilon}, N,1}_t, \cdots, Z^{\psi^{\epsilon}, N,N}_t\right), t\in[0,T]\right\}$ be the unique solution of the SDEs \eqref{Pertur-Hawkes-psi-eq}  associated with $\psi^{\epsilon}$ and let $ \mathbb{P}^{\psi^{\epsilon}, N}$ denote its  probability distribution. 
Then we can compute that
\begin{equation*} \label{N-dim-perturbation-Hawkes-process-NR-eq}
\begin{aligned}
\log\frac{d\mathbb P^{\psi^{\epsilon}, N}}{d\mathbb P^{0,N}}&=\sum_{i=1}^N\int_0^T \int_0^\infty \left( \frac{a(N)}{\sqrt{N}}\nabla\psi^{\epsilon}\left(s,Z^{\psi^{\epsilon},N,i}_{s-}\right)\right)I_{\left\{z \leq \phi \left( \int_0^{s-}h(s-u)d\overline{Z}^{\psi,N}_u\right)\right\}}
\pi^i(ds\,dz) \\
&\quad -\sum_{i=1}^N  \int_0^T   \bigg(e^{\frac{a(N)}{\sqrt{N}}\nabla\psi^{\epsilon}\left(s,Z^{\psi^{\epsilon},N,i}_{s}\right)}-1\bigg)  \phi \left( \int_0^{s}h(s-u)d\overline{Z}^{\psi,N}_u\right)ds.
\end{aligned}
\end{equation*}
Then from Theorem~\ref{Pertur-Hawkes-LLN-thm}, we have
\begin{align}\label{MDP-New-LLN-1}
\lim_{N\rightarrow\infty}\mathbb{P}^{\psi^{\epsilon},N}\left(\widetilde{L}^{N}\in O\right)
=\lim_{N\rightarrow\infty}\mathbb{P}\left(\widetilde{L}^{N,\psi^{\epsilon}}\in O\right)=1.
\end{align}
By It\^{o}'s formula,
\begin{equation*}
\begin{aligned}
&a^2(N)\bigg(\left\<\widetilde{L}^N_T,\psi^{\epsilon}(T)\right\>-\int_0^{T} \left\< \widetilde{L}^{N},\partial_t \psi^{\epsilon}(t)\right\>   dt- \int_0^{T} \left\<\widetilde{L}_t^N, \nabla \psi^{\epsilon}(t)\right\>  \phi\left(\int_0^{t} h(t-s)d m_s\right) dt\\
&-\int_0^{T} \left\<{L}_t^N, \nabla \psi^{\epsilon}(t)\right\>   \frac{\sqrt{N}}{a(N)}\left( \phi\left(\int_0^{t} h(t-s)d \overline{Z}^{N}_s\right)-\phi\left(\int_0^{t} h(t-s)d m_s\right)\right) dt\bigg)\\
&=\sum_{i=1}^N\int_0^{T} \int_0^\infty \frac{a(N)}{\sqrt{N}}\nabla  \psi^{\epsilon}\left(t,Z_{t-}^{N,i}\right)I_{ \left\{z \leq \phi\left(\int_0^{t-}
h(t-s)d \overline{Z}^{N}_s\right)\right\}}   ({\pi}^{i}(dt\,dz)-dtdz)\,.
\end{aligned}
\end{equation*}
Then we can compute that
\begin{equation}\label{MDP-lower-bound-eq}
\begin{aligned}
&\log\frac{d\mathbb P^{\psi^{\epsilon},N}}{d\mathbb P^{0,N}}\\
&=a^2(N)\bigg(\left\<\widetilde{L}^N_T,\psi^{\epsilon}(T)\right\>-\int_0^{T} \left\< \widetilde{L}^{N},\partial_t \psi^{\epsilon}(t)\right\>   dt- \int_0^{T} \left\<\widetilde{L}_t^N, \nabla \psi^{\epsilon}(t)\right\>  \phi\left(\int_0^{t} h(t-s)d m_s\right) dt\\
&-\int_0^{T} \left\<{L}_t^N, \nabla \psi^{\epsilon}(t)\right\>   \frac{\sqrt{N}}{a(N)}\left( \phi\left(\int_0^{t} h(t-s)d \overline{Z}^{N}_s\right)-\phi\left(\int_0^{t} h(t-s)d m_s\right)\right) dt\bigg)\\
&-\sum_{i=1}^N  \int_0^T   \bigg(e^{\frac{a(N)}{\sqrt{N}}\nabla\psi^{\epsilon}\left(s,Z^{N,i}_{s}\right)}-1-\frac{a(N)}{\sqrt{N}}\nabla\psi^{\epsilon}\left(s,Z^{N,i}_{s}\right)\bigg)  \phi \left( \int_0^{s}h(s-u)d\overline{Z}_u^{N}\right)ds.
\end{aligned}
\end{equation}
Set
\begin{equation} 
\begin{aligned}
A_1:=& \bigg\{\bigg|\int_0^T\left\<  L_t^N, \nabla \psi^{\epsilon}(t)\right\>  \frac{\sqrt{N}}{a(N)}\bigg( \phi\left(\int_0^{t} h(t-s)d \overline{Z}^{N}_s\right) -\phi\left(\int_0^{t} h(t-s)d m_s\right)\bigg)dt\\
&\qquad\qquad-\int_0^T\left\< \mathcal L_t, \nabla \psi^{\epsilon}(t)\right\>  \phi'\bigg(\int_0^{t} h(t-s)d m_s\bigg)  \int_0^{t} h(t-s) d\left\<\widetilde{L}_s^N,\ell\right\>dt\bigg|\geq \epsilon \bigg\},
\end{aligned}
\end{equation}
and
\begin{equation} 
\begin{aligned}
A_2:= &\bigg\{\bigg|\frac{N}{a^2(N)} \int_0^T\bigg\<  {L}_t^N,e^{\frac{a(N)}{\sqrt{N}}\nabla \psi^{\epsilon}(t)}-1-\frac{a(N)}{\sqrt{N}}\nabla \psi^{\epsilon}(t)\bigg\>  \phi\left(\int_0^{t-}
h(t-s)d \overline{Z}^{N}_s\right)  \\
&\qquad\qquad   - \frac{1}{2}\left\<  \mathcal L_t, (\nabla \psi^{\epsilon}(t))^2\right\>  \phi\left(\int_0^{t}
h(t-s)d m_s\right) dt \bigg|\geq \epsilon \bigg\}.
\end{aligned}
\end{equation}
Then
$$
\begin{aligned}
&~~\frac{1}{a^2(N)}\log \mathbb P\left(\widetilde{L}^{N}\in O\right)\\
&=\frac{1}{a^2(N)}\log \frac{1}{\mathbb P^{\psi^{\epsilon},N}\left(\widetilde{L}^{N}\in O\right)}
\mathbb E^{\psi^{\epsilon},N}\left(I_{\{\widetilde{L}^{N}\in O\}} \frac{d\mathbb P^{0, N}}{d\mathbb P^{\psi^{\epsilon},N}}\right)
+\frac{1}{a^2(N)}\log \mathbb P^{\psi^{\epsilon},N}\left(\widetilde{L}^{N}\in O\right)\\
&{\geq} \frac{1}{a^2(N)}
\mathbb E^{\psi^{\epsilon},N}\left(I_{\{\widetilde{L}^{N}\in O\}} \log\frac{d\mathbb P^{0, N}}{d\mathbb P^{\psi^{\epsilon},N}}\right)
+\frac{1}{a^2(N)}\log \mathbb P^{\psi^{\epsilon},N}\left(\widetilde{L}^{N}\in O\right)\\
&\geq \frac{1}{a^2(N)}
\mathbb E^{\psi^{\epsilon},N}\left(I_{\{\widetilde{L}^{N}\in O\}\cap A_1^c\cap A_2^c} \log\frac{d\mathbb P^{0, N}}{d\mathbb P^{\psi^{\epsilon},N}}\right)
+\frac{1}{a^2(N)}\log \mathbb P^{\psi^{\epsilon},N}\left(\widetilde{L}^{N}\in O\right),
\end{aligned}
$$
where $\mathbb E^{\psi^{\epsilon},N}$ is the associated expectation of $\mathbb P^{\psi^{\epsilon},N}$. By \eqref{MDP-New-LLN-1} and \eqref{MDP-lower-bound-eq}, we have 
 $$
\begin{aligned}
&\liminf_{N\to\infty}\frac{1}{a^2(N)}\log \mathbb P\left(\widetilde{L}^{N}\in O\right)\\
\geq&- \left({\Upsilon}_{\mu^{\epsilon}}(\psi^{\epsilon})-\frac{1}{2}\int_0^T\left\< \mathcal L_t, (\nabla \psi^{\epsilon}(t))^2 \right\>  \phi\left(\int_0^{t}
h(t-s)d m_s\right)dt\right)-2\epsilon\\
\geq &-I(\mu)-3\epsilon,
\end{aligned}
$$
where the last equality hold by  \eqref{MDP-rate-f-property-lem-eq-2} and Lemma~\ref{MDP-rate-f-property-lem}.  Letting $\epsilon\to0$, we obtain 
 $$
\begin{aligned}
\liminf_{N\to\infty}\frac{1}{a^2(N)}\log \mathbb P\left(\widetilde{L}^{N}\in O\right)\geq-I(\mu).
\end{aligned}
$$
This completes the proof of the lower bound.
\end{proof}

\section{Proofs of Two Corollaries}\label{sec:cor}

In this section, we show that our results (Theorem~\ref{CLT-thm} and Theorem~\ref{MDP-thm}) can recover 
the fluctuations (Corollary~\ref{CLT-thm-mean-p}) and moderate deviations (Corollary~\ref{MDP-thm-mean-p}) for the mean process obtained in \cite{GaoFZhu2018SPA}.

First, let us prove Corollary~\ref{CLT-thm-mean-p}, which establishes the fluctuations for the mean process.

\begin{proof}[Proof of Corollary~\ref{CLT-thm-mean-p}]
First of all, note that the mapping $D([0,T], \mathbb H^{-1}(\mathbb N)) \ni \mu\to  \<\mu,\ell\>\in D([0,T], \mathbb R) $ is continuous. By the continuous mapping theorem in the weak convergence, we recover the fluctuation theorem   for mean process (which was first obtained in Gao and Zhu \cite{GaoFZhu2018SPA} and  Heesen and Stannat \cite{HeesenStannat2021SPA}).  
In fact,   we take $\varphi=\ell$  in \eqref{CLT-thm-eq} for any $0\leq t\leq T$ instead of at $T$, and then using $\<X_s,\nabla\ell\>=0$, $\<\mathcal L_s,\nabla\ell\>=\<\mathcal L_s,(\nabla\ell)^2\>=1$,  and noting  that  $W_s:=\<\sqrt{\mathcal L_s}, B_s\>=\sum_{x\in\mathbb N}\sqrt{\mathcal L_s(x)}B_s(x) $ is a standard Brownian motion,  we have established
Corollary~\ref{CLT-thm-mean-p}.
\end{proof}

Next, let us prove Corollary~\ref{MDP-thm-mean-p}, which establishes the moderate deviations for the mean process.

\begin{proof}[Proof of Corollary~\ref{MDP-thm-mean-p}]
First of all, note that the mapping $D([0,T], \mathbb H^{-1}(\mathbb N)) \ni \mu\to  \<\mu,\ell\>\in D([0,T], \mathbb R) $ is continuous. By contraction principle (see e.g. \cite{DemboZeitouni1998Book}),  we know that $\mathbb{P}\left(\left\{\frac{\sqrt{N}(\overline{Z}_{t}^{N}-m_t)}{a(N)},0\leq t\leq T\right\}\in\cdot\right)$ satisfies a large deviation principle with the speed $a^{2}(N)$ and the rate function
\begin{equation}\label{plug:into-rate-J-eq-1}
J(\eta)=\inf_{\mu:\<\mu_{t},\ell\>=\eta_{t},0\leq t\leq T}I(\mu).
\end{equation}
Next, let us give a precise formula for $J(\eta)$.   From the proof of Theorem~\ref{MDP-thm} (see subsection \ref{sec:perturbed}), for any $\mu$ with  $I(\mu)<\infty$,
there exists some measurable functions $\psi $  on  $[0,T]\times \mathbb N$ such that   
$$
{I}(\mu)=\frac{1}{2}\int_{0}^{T}\left\langle\mathcal{L}_{t},(\nabla\psi(t))^{2}\right\rangle\phi\left(\int_{0}^{t}h(t-s)dm_{s}\right)dt<\infty,
$$
and  for any $ \varphi\in C^{1,lip}([0,T]\times\mathbb{N})$,   $t\in (0,T]$
\begin{equation}\label{plug:into-eq-1}
\begin{aligned}
&\< \mu_t,\varphi(t)\>- \int_0^{t}\< \mu_{s},\partial_s \varphi(s)\>   ds-\int_0^{t}\< \mu_s, \nabla \varphi(s)\>   \phi\left(\int_0^{s} h(s-u)d m_u\right)  ds\\
&\qquad-\int_0^{t}\< \mathcal L_s, \nabla \varphi(s)\> \phi'\left(\int_0^{s}h(s-u)d m_u\right)  \int_0^{s}h(s-u)d \<\mu_u,\ell\>ds\\
&\qquad\qquad-\int_0^t\<\nabla\psi(s) \mathcal L_s, \nabla \varphi(s) \>  \phi\left(\int_0^{s}
h(s-u)d m_u\right) ds=0.
\end{aligned}
\end{equation}
If  $J(\eta)<\infty$, then  for  any $\epsilon>0$,  there exists  $\mu$ with  $I(\mu)<\infty$ such that  $\<\mu_{s},\ell\>=\eta_{s}$ for any $0\leq s\leq T$ and 
$$
J(\eta)\geq I(\mu)-\epsilon.
$$
Now,  by taking  $\varphi=\ell$ in \eqref{plug:into-eq-1} and applying $\<\mu_{s},\ell\>=\eta_{s}$ for any $0\leq s\leq T$
and $\langle\mathcal{L}_{s},\nabla\ell\rangle=1$, 
we get
\begin{equation*}
\eta_{t}-\int_{0}^{t}\phi'\left(\int_0^{s}
h(s-u)d m_u\right)  \int_0^{s}
h(s-u)d \eta_{u}ds
=\int_{0}^{t}\langle\mathcal{L}_{s},\nabla\psi(s)\rangle\phi\left(\int_{0}^{s}h(s-u)dm_{u}\right)ds,
\end{equation*}
which yields
\begin{equation}\label{psi:equation-eq-1}
\langle\mathcal{L}_{t},\nabla\psi(t)\rangle=\frac{\partial_t\eta_{t}-\phi'\left(\int_0^{t}
h(t-s)d m_s\right)  \int_0^{t}h(t-s)d \eta_{s}}{\phi\left(\int_{0}^{t}h(t-s)dm_{s}\right)}.
\end{equation}
Therefore, by the Cauchy-Schwarz inequality, 
$$
\begin{aligned}
J(\eta)+\epsilon\geq I(\mu)=&\frac{1}{2}\int_{0}^{T}\langle\mathcal{L}_{t},(\nabla\psi(t))^{2}\rangle\phi\left(\int_{0}^{t}h(t-s)dm_{s}\right)dt\\
\geq &\frac{1}{2}\int_{0}^{T}\frac{\left(\partial_t\eta_{t}-\phi'\left(\int_0^{t}
h(t-s)d m_s\right)  \int_0^{t}h(t-s)d \eta_{s}\right)^2}{\phi\left(\int_{0}^{t}h(t-s)dm_{s}\right)}dt,
\end{aligned}
$$
and thus
\begin{equation}\label{plug:into-rate-J-eq-2}
\begin{aligned}
J(\eta) \geq & \frac{1}{2}\int_{0}^{T}\frac{\left(\partial_t\eta_{t}-\phi'\left(\int_0^{t}
h(t-s)d m_s\right)  \int_0^{t}h(t-s)d \eta_{s}\right)^2}{\phi\left(\int_{0}^{t}h(t-s)dm_{s}\right)}dt.
\end{aligned}
\end{equation}
On the other hand,  we can take  
$$
\psi^*(t,x)=\frac{\left(\partial_t\eta_{t}-\phi'\left(\int_0^{t}
h(t-s)d m_s\right)  \int_0^{t}h(t-s)d \eta_{s}\right)\ell(x)}{\phi\left(\int_{0}^{t}h(t-s)dm_{s}\right)},
$$
and let $\mu^*$ be the unique solution of the  equation \eqref{plug:into-eq-1} with $\psi=\psi^*$.  Set $\eta_t^*=\<\mu_t^*, \ell\>$.  Then, we have
\begin{align*}
\frac{\partial_t\eta_{t}-\phi'\left(\int_0^{t}
h(t-s)d m_s\right)  \int_0^{t}h(t-s)d \eta_{s}}{\phi\left(\int_{0}^{t}h(t-s)dm_{s}\right)}
&=\langle\mathcal{L}_{t},\nabla\psi(t)\rangle=\langle\mathcal{L}_{t},\nabla\psi^*(t)\rangle\\
&=\frac{\partial_t\eta_{t}^*-\phi'\left(\int_0^{t}
h(t-s)d m_s\right)  \int_0^{t}h(t-s)d \eta_{s}^*}{\phi\left(\int_{0}^{t}h(t-s)dm_{s}\right)}.
\end{align*}
Therefore,
$\eta^*=\eta$, and
\begin{equation}\label{plug:into-rate-J-eq-3}
\begin{aligned}
J(\eta)=J(\eta^*)\leq I(\mu^*)=&\frac{1}{2}\int_{0}^{T}\langle\mathcal{L}_{t},(\nabla\psi^*(t))^{2}\rangle\phi\left(\int_{0}^{t}h(t-s)dm_{s}\right)dt\\
=& \frac{1}{2}\int_{0}^{T}\frac{\left(\partial_t\eta_{t}-\phi'\left(\int_0^{t}
h(t-s)d m_s\right)  \int_0^{t}h(t-s)d \eta_{s}\right)^2}{\phi\left(\int_{0}^{t}h(t-s)dm_{s}\right)}dt.
\end{aligned}
\end{equation}
This completes the proof.
\end{proof}


\end{document}